\documentclass{amsart}
\usepackage{subfigure}
\usepackage{graphicx,verbatim, amsmath, amssymb, amsfonts, ifthen,mathtools,mathabx,enumerate,epstopdf,placeins,longtable, float, placeins, hyperref, tikz, enumitem}	
\DeclareGraphicsExtensions{.ps}
\newtheorem{proposition}{Proposition}[section]
\newtheorem{theorem}[proposition]{Theorem}
\newtheorem{corollary}[proposition]{Corollary}
\newtheorem{lemma}[proposition]{Lemma}

\newtheorem{example}[proposition]{Example}

\newtheorem{question}[proposition]{Question}
\newtheorem{remark}[proposition]{Remark}

\numberwithin{equation}{section}

\newcommand{\newword}[1]{\textbf{\emph{#1}}}

\newcommand{\poset}{\mathcal{P}}

\renewcommand{\Join}{\bigvee}
\newcommand{\Meet}{\bigwedge}

\renewcommand{\th}{^{\text{th}}}
\newcommand{\meet}{\wedge}
\newcommand{\join}{\vee}
\newcommand{\A}{\mathcal{A}}
\newcommand{\E}{\mathcal{E}}

\newcommand{\G}{\mathcal{G}}

\newcommand{\K}{\mathcal{K}}
\newcommand{\T}{\mathcal{T}}

\newcommand{\pidown}{\pi_\downarrow}
\newcommand{\piup}{\pi^\uparrow}
\newcommand{\covers}{{\,\,\,\cdot\!\!\!\! >\,\,}}

\newcommand{\1}{\hat{1}}
\newcommand{\0}{\hat{0}}
\newcommand{\2}{\mathbf{2}}

\newcommand{\refines}{{\,\underline{\ll}\,}}
\newcommand{\can}{\operatorname{can}}
\newcommand{\cov}{\operatorname{cov}}
\newcommand{\Irr}{\operatorname{Irr}}
\newcommand{\ijr}{\operatorname{ijr}}

\newcommand{\covdown}{\cov_{\downarrow}}
\newcommand{\covup}{\cov^{\uparrow}}
\newcommand{\wedgesum}{\mathbin{\rotatebox[origin=c]{270}{$\succ$}}}
\newcommand{\Row}{\operatorname{Row}}
\newcommand{\tors}{\operatorname{tors}}

\author{Emily Barnard}
\title{The Canonical Join Complex}
\address{Department of Mathematics, North Carolina State University, Raleigh, NC, USA}

\begin{document}
\maketitle

\begin{abstract}
In this paper, we study the combinatorics of a certain minimal factorization of the elements in a finite lattice $L$ called the canonical join representation.
The join $\Join A =w$ is the canonical join representation of $w$ if $A$ is the unique lowest subset of $L$ satisfying $\Join A=w$ (where ``lowest'' is made precise by comparing order ideals under containment).
When each element in $L$ has a canonical join representation, we define the canonical join complex to be the abstract simplicial complex of subsets $A$ such that $\Join A$ is a canonical join representation.
We characterize the class of finite lattices whose canonical join complex is flag, and show how the canonical join complex is related to the topology of $L$.
\end{abstract}

\setcounter{tocdepth}{2}
\tableofcontents
\section{Introduction}
In a finite lattice $L$, the canonical join representation of an element $w$ is a certain unique minimal factorization of $w$ in terms of the join operation.
Specifically, the join-representation $\Join A = w$ is the \newword{canonical join representation} of $w$ if the join $\Join A$ is irredundant and the set $A$ is taken as low as possible in the partial order on $L$.
(See Section~\ref{background} for the precise definition.) 
There is an analogous factorization in terms of the meet operation called the \newword{canonical meet representation} that is defined dually (replacing ``$\Join$'' with ``$\Meet$'' and  ``lowest'' with ``highest'' in the sentence above).
The canonical join representation or canonical meet representation for a given element may not exist.
See Figure~\ref{fig:meet_semi} for two examples.
If each element in $L$ has a canonical join representation then $L$ is \newword{join-semidistributive}.
We say that $L$ is \newword{semidistributive} if each element \textit{also} has a canonical meet representation.
(See Section~\ref{background} and Theorem~\ref{join_semi_cjr} in particular for an equivalent definition.) 

When $L$ is join-semidistributive, we define the \newword{canonical join complex} to be the abstract simplicial complex whose faces are the subsets $A\subset L$ such that the join $\Join A$ is a canonical join representation.
(Proposition~\ref{simplicial complex} says that this is indeed a complex.)
We define the \newword{canonical meet complex} similarly.
Recall that a simplicial complex is \newword{flag} if it is the clique complex of its 1-skeleton, or equivalently, its minimal non-faces have size two.
Our main result is:

\begin{theorem}\label{flag}
Suppose $L$ is a finite join-semidistributive lattice.
Then the canonical join complex of $L$ is flag if and only if $L$ is semidistributive.
\end{theorem}
In other words, if each element in $L$ admits a canonical join representation, then the canonical join complex for $L$ is flag if and only if each element also admits a canonical meet representation.
In light of Theorem~\ref{flag}, we define the \newword{canonical join graph} for $L$ to be the one-skeleton of its canonical join complex.
Canonical join representations and the canonical join graph appear in many familiar guises.
See Section~\ref{Examples} for connections to comparability graphs and noncrossing partitions.

It is not hard to find examples of finite join-semidistributive lattices whose canonical join complex is not flag.
A key example is shown below in Figure~\ref{fig:join_semi}.
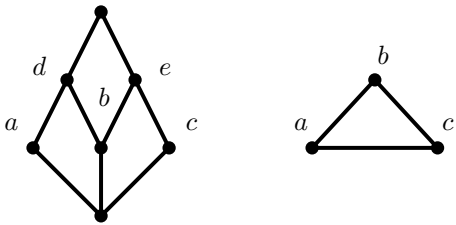
\begin{figure}[h]
  \centering
   \scalebox{1}{ \includegraphics{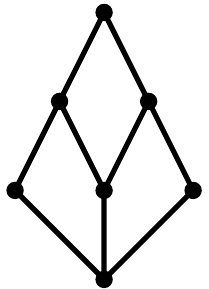} \qquad\qquad
    \includegraphics{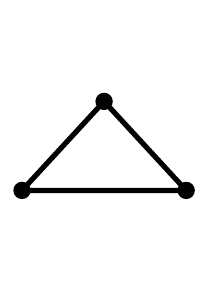}
  \begin{picture}(0,50)(0,50)
   \put(-105,85){$c$}
   \put(-138,94){$b$}
   \put(-173.5,85){$a$}
   \put(-163,106){$d$}
   \put(-115,106){$e$}
   \put(-8,85){$c$}
   \put(-32.4,110){$b$}
   \put(-64,85){$a$}
   \end{picture}}
     \caption{The canonical join complex is an empty triangle.}
        \label{fig:join_semi}
\end{figure}
Observe that each pair of atoms in the lattice in Figure~\ref{fig:join_semi} is a face in the canonical join complex.
Since the join of all three atoms is redundant (because we can remove $b$ and obtain the same join), the canonical complex is an empty triangle.
Note that the bottom element $\0$ of this lattice does not have a canonical meet representation:
Both $a\meet e$ and $c\meet d$ are minimal, highest meet-representations for $\0$.
We will see below that the combinatorics of the canonical join complex (and canonical meet complex) are closely related to the topology of the its lattice.

Recall that the \newword{crosscut complex} of $L$ is the abstract simplicial complex whose faces are the subsets $A'$ of atoms in $L$ such that $\Join A' < \1$.
A lattice is \newword{crosscut-simplicial} if for each interval $[x,y]$ the join of each proper subset of atoms in $[x,y]$ is strictly less than $y$.
Recall that the order complex of a finite poset $\poset$ is homotopy equivalent to its crosscut complex (\cite[Theorem~10.8]{Btop}).
Therefore, if $L$ is crosscut-simplicial then each interval $[x,y]$ in $L$ is either contractible or homotopy equivalent to a sphere with dimension two less than the number of atoms in $[x,y]$ (see also \cite[Theorem~3.7]{SB-labeling}).
In particular, $\mu(x,y) \in \{-1,0,1\}$.

Observe that the facets of the crosscut complex for the lattice $L$ in Figure~\ref{fig:join_semi} are $\{a,b\}$ and $\{b,c\}$.
Therefore, $L$ is not crosscut-simplicial.
By contrast, Hersh and M\'{e}sz\'{a}ros recently showed that a large class of finite semidistributive lattices---including the class of finite distributive lattices, the weak order on a finite Coxeter group, and the Tamari lattice (\cite[Theorems~5.1, 5.3 and 5.5]{SB-labeling})---are crosscut-simplicial.
Building on this work, McConville proved that if $L$ is semidistributive, then it is crosscut-simplicial (\cite[Theorem~3.1]{crosscut}).
When each element in $L$ has a canonical join representation, we prove that the converse is true. 

\begin{theorem}\label{crosscut converse}
Suppose that $L$ is a finite join-semidistributive lattice.
The following are equivalent:
\begin{enumerate}
\item The canonical join complex for $L$ is flag.
\item $L$ is crosscut-simplicial.
\item $L$ is semidistributive.
\end{enumerate}
\end{theorem}

As an immediate corollary, we obtain the following topological obstruction to the flag-property of the canonical join complex.
\begin{corollary}\label{top}
Suppose that $L$ is a finite join-semidistributive lattice and its canonical join complex is flag.
Then:
\begin{enumerate}
\item Each interval $[x,y]$ in $L$ is either contractible or homotopy equivalent to $\mathbb{S}^{d-2}$, where $d$ is the number of atoms in $[x,y]$;
\item The M\"obius function takes only the values $\{-1,0,1\}$ on the intervals of $L$.
\end{enumerate}
\end{corollary}
McConville showed in \cite[Corollary~5.4]{crosscut} that if $L$ is crosscut-simplicial then so is each of its lattice quotients.
Because semidistributivity is preserved under taking sublattices and quotients \textit{when $L$ is finite} (see Section~\ref{subs}), we immediately obtain the following extension of McConville's result for finite join-semidistributive lattices.
\begin{corollary}
Suppose that $L$ is a finite join-semidistributive lattice that is crosscut-simplicial.
Then each sublattice and quotient lattice of $L$ is also crosscut-simplicial.
\end{corollary}

Theorem~\ref{flag} is surprising in part because its proof does not explicitly use the canonical meet representation of the elements in $L$.
Instead, we make use of local characterization of canonical join representations in terms of the cover relations, and a bijection $\kappa$ from the join-irreducible to the meet-irreducible elements in $L$.
As an easy consequence of this approach, we obtain the following nice result:

\begin{corollary}\label{iso}
Suppose that $L$ is a finite semidistributive lattice.
Then the bijection $\kappa$ induces an isomorphism from canonical join complex to the canonical meet complex of $L$. 
\end{corollary}

Using the isomorphism from Corollary~\ref{iso}, one obtains an operation on the canonical join complex that generalizes the operation of rowmotion (on the set of antichains in a poset) and the operation of Kreweras complementation (on the set of noncrossing partitions).
See Remark~\ref{rowmotion}.

The canonical join complex was first introduced in \cite{arcs}, in which Reading showed that it is flag for the special case of the weak order on the symmetric group (see Example~\ref{coxeter}).
Recently, canonical join representations have played a role in the study of functorially finite torsion classes for the preprojective algebra of Dynkin-type $W$, when $W$ is a simply laced Weyl group (see for example \cite{G-M, IRRT}).
In the forthcoming \cite{B-C-Z}, the authors study the canonical join complex for \textit{any} finite dimensional associative algebra $\Lambda$ of finite representation type.
Since the weak order on any finite Coxeter group $W$ and the lattice of torsion classes for $\Lambda$ of finite representation type are both examples of finite semidistributive lattices (see \cite[Lemma~9]{weak order} and \cite[Theorem~4.5]{G-M}), we obtain the following two applications of Theorem~\ref{flag}:

\begin{corollary}
Suppose that $W$ is a finite Coxeter group.
Then the canonical join complex for the weak order on $W$ is flag.
\end{corollary}


\begin{corollary}
Suppose that $\Lambda$ is an associative algebra of finite representation type, and $\tors(\Lambda)$ is its lattice of torsion classes ordered by containment.
Then the canonical join complex for $\tors(\Lambda)$ is flag.
\end{corollary}

\section{Motivation and Examples}\label{Examples}
Before we give the technical background for our main results, we describe several familiar examples in which the combinatorics of canonical join representations appear.
We begin with an example from number theory and commutative algebra.
\begin{example}[The divisibility poset]\label{divisibility poset}
\normalfont 
It is often useful to give a canonical factorization of the elements in a set of equipped with some algebraic structure.
A familiar example from number theory and commutative algebra is the primary decomposition of ideals.
The canonical join representation is the natural lattice-theoretic analogue. 
Indeed, when $L$ is the the divisibility poset (whose elements are the positive integers ordered $r\le s$ if and only if $r|s$), the canonical join representation of $x\in L$ coincides with the primary decomposition of the ideal generated by $x$:
 \[x=\Join \{p^d: \text{$p$ is prime and $p^d$ is the largest power of $p$ dividing $x$}\}.\]
 \end{example}
 
Suppose that $L$ is a finite lattice, such that each element in $L$ admits a canonical join representation. 
One pleasant property of the canonical join representation (and its dual, the canonical meet representation) is that it ``sees'' the geometry the Hasse diagram for $L$.
Suppose that $w\in L$ has the canonical join representation $\Join A$.
We will shortly prove that the factors that appear in $A$ are naturally in bijection with the elements covered by $w$.
So, the down-degree of $w$ is equal to the size of $A$.
Specifically, we will prove the following proposition (see Lemma~\ref{canonical cover} and Proposition~\ref{canonical joinands}):
\begin{proposition}\label{canonical joinands first time}
Suppose that $\Join A=w$ is a face in the canonical join complex for $L$.
Then, for each element $y$ that is covered by $w$ there is a corresponding element $j\in A$ such that $j\join y=w$, and $j$ is the unique minimal element in $L$ with this property.
The correspondence $y\mapsto j$ is a bijection.
\end{proposition}  
With this proposition in mind, we consider the class of finite distributive lattices.
\begin{example}[Finite distributive lattices]\label{distributive lattices}
\normalfont
Suppose that $L$ is a finite distributive lattice.
Recall that the fundamental theorem of finite distributive lattices (see for example \cite[Theorem~3.4.1]{EC1}) says that $L$ is the lattice $J(\poset)$ of order ideals of some finite poset $\poset$.
Suppose that $A$ is an antichain in $\poset$.
We write $I_A$ for the order ideal generated by $A$ (that is, the elements of $A$ are the maximal elements of $I_A$).
Dually, we write $I^A$ for the order ideal satisfying: $A$ is the set of minimal elements in $\poset\setminus I^A$.
Observe that the order ideals covered by $I_A$ are exactly of the form $I_{A\setminus \{y\}}=I_A\setminus \{y\}$, where $y\in A$.
Since $I_y$ is the smallest order ideal in $J(\poset)$ containing $y$, it follows immediately from Proposition~\ref{canonical joinands first time} that the canonical join representation of $I_A$ is $\bigcup \{I_y: y\in A\}$.
(Dually, the canonical meet representation for the ideal $I^A$ is $\bigcap \{I^y: y\in A\}$.)
It follows that the canonical join graph of $J(\poset)$ is the incomparability graph of $\poset$.

Comparability graphs were classified by a theorem of Gallai which we quote from \cite[Theorem~2.1]{trotter} below.
\begin{theorem}\label{comparability_graph}
A graph $G$ is a comparability graph for a finite poset if and only if it does not contain as an induced subgraph any graph from \cite[Table 1]{trotter} or the complement of any graph appearing in \cite[Table 2]{trotter}.
\end{theorem}
As an immediate corollary we have the following characterization of the canonical join graphs for finite distributive lattices.
Each finite distributive lattice is, in particular, semidistributive.
By Theorem~\ref{flag}, we obtain a complete characterization of the canonical join complexes for finite distributive lattices.
\begin{proposition}\label{canonical_graph}
The graph $G$ is the canonical join graph for a finite distributive lattice if and only if $G$ does not contain, as an induced subgraph, the complement of any graph forbidden by Theorem~\ref{comparability_graph}.
\end{proposition}
\end{example}

\begin{example}[The Symmetric group and noncrossing diagrams]\label{coxeter}
\normalfont
Recently Reading gave an explicit combinatorial model for the canonical join complex of the weak order on the symmetric group $S_n$ in terms of certain noncrossing diagrams.
A \newword{noncrossing diagram} is a diagram consisting of $n$ vertices arranged vertically, together with a collection of curves called arcs that must satisfy certain compatibility conditions.
In particular, the arcs in a noncrossing diagram do not intersect in their interiors (see \cite{arcs} for details).
Each diagram is determined by its combinatorial data: the endpoints of its arcs, and on which side (either left or right) each arc passes the vertices in the diagram.
For example, a we say that a diagram contains only \newword{left arcs} if it has no arc that passes to the right of any vertex (see the leftmost noncrossing diagram in Figure~\ref{fig:left arcs}).
\begin{figure}[h]
  \centering
   \scalebox{1}{ \includegraphics{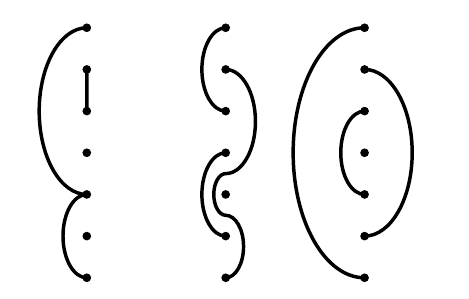}}
     \caption{Some examples of noncrossing diagrams.}
        \label{fig:left arcs}
\end{figure}

We say that two arcs are \newword{compatible} if there is a noncrossing diagram that contains them.
The following is a combination of \cite[Corollary~3.4 and Corollary~3.6]{arcs}.
(In the statement of the Theorem, we take ``a collection of arcs'' to also mean collection of noncrossing diagrams, each containing a single arc.)
\begin{theorem}\label{faces and arcs} 
There is a bijection $\delta$ from the set of join-irreducible permutations in $S_n$ to the set of noncrossing diagrams on $n$ vertices with a unique arc.
Moreover, a collection of arcs $\mathcal{E}$ corresponds to a face in the canonical join complex for $S_n$ if and only if the arcs in $\mathcal{E}$ are pairwise compatible.
\end{theorem}

\end{example}

\begin{example}[The Tamari lattice and noncrossing partitions]\label{noncrossing}
\normalfont
We conclude our list of examples by considering the Tamari lattice.
The Tamari lattice $\T_n$ is a finite semidistributive lattice (see for example \cite[Theorem~3.5]{tamari}), which can be realized as an ordering on the set of triangulations for a fixed convex $(n+3)$-gon $P_n$.
Recall that the rank $n$ associahedron is a simple convex polytope, whose faces are in bijection with the collections of pairwise noncrossing diagonals of $P_n$ (see~\cite[Figure~3.5]{root systems}).
In particular, its vertices are parametrized by triangulations of $P_n$ in such a way that we obtain the Hasse diagram for $\T_n$ as an orientation of its 1-skeleton.
Since the number of factors in a canonical join representation (called the \newword{canonical joinands}) for $w\in \T_n$ is equal to the down-degree of $w$, we obtain the following result:
\begin{proposition}\label{h and f vectors}
The $f$-vector for the canonical join complex of the Tamari lattice $\T_n$ is equal to the the $h$-vector of the rank $n$ associahedron.
Specifically, the number of size-$k$ faces in the canonical join complex is equal to the Narayana number \[N(n,k) = \frac{1}{n+1}\binom{n+1}{k+1}\binom{n+1}{k}.\]
\end{proposition}

Indeed, the canonical join representation of $w\in \T_n$ is essentially a noncrossing partition.
It is well-known that the Tamari lattice $\T_n$ may be realized as the set of permutations avoiding the 231-pattern.
It is a fact that a permutation avoids the 231-pattern if and only if its image under the bijection $\delta$ from Theorem~\ref{faces and arcs} is a noncrossing arc diagram consisting of only left arcs.
Rotating such a diagram by a quarter-turn gives the familiar representation of a noncrossing partition as a bump diagram.
(See \cite[Example~4.5]{arcs} for details, and \cite[Theorem~2.7]{typefree} and the discussion following \cite[Proposition~8.8]{typefree} for a type-free discussion.)
\end{example}

\section{Finite semidistributive lattices}\label{foundations}
\subsection{Definitions}\label{background}
In this paper, we study only finite lattices.
We write $\0$ for the unique smallest element in $L$ and $\1$ for the unique largest element.
A \newword{join-representation} of $w$ is an expression $\Join A$ which evaluates to $w$ in $L$.
At times we will also refer to the set $A$ as a join-representation.
We write $\covdown(w)$ for the set $\{y\in L: w\covers y\}$.
Similarly, we write $\covup(w)$ for the set of upper covers of $w$.
Recall that $w$ is \newword{join-irreducible} if $w=\Join A$ implies that $w\in A$.
(In particular, the bottom element $\0$ is not join-irreducible, because it is equal to the empty join.)
Since $L$ is finite, $w$ is join-irreducible when $\covdown(w)$ has exactly one element.
\newword{Meet-irreducible} elements satisfy the dual condition.
We write $\Irr(L)$ for the set of join-irreducible elements of $L$.

A join-representation $\Join A$ of $w$ is \newword{irredundant} if $\Join A' < \Join A$ for each proper subset $A'\subset A$.
Each irredundant join-representation is an antichain in $L$.
We say that the subset $A$ of $L$ \newword{join-refines} a subset $B$ if for each element $a$ in $A$, there exists some element $b$ in $B$ such that $a\le b$.
Join-refinement defines a preorder on the subsets of $L$ that is a partial order (corresponding to the containment of order ideals) when restricted to the set of antichains in $L$.

We write $\ijr(w)$ for the set of irredundant join-representations of $w$.
The \newword{canonical join representation} of $w$ in $L$, when it exists, is the unique minimal element, in the sense of join-refinement, of $\ijr(w)$.
We write $\can(w)$ for the canonical join representation of $w$.
An element $j\in \can(w)$ is a \newword{canonical joinand} for $w$.
If  $A=\can(w)$, we say that $A$ \newword{joins canonically}.
It follows immediately from the definition that each canonical joinand of $w$ is join-irreducible.
Moreover, the canonical join representation of each join-irreducible element $j$ exists and is equal to $\{j\}$.
The canonical meet representation of $w$ (when it exists) is defined dually.

In Figure~\ref{fig:meet_semi}, we give two examples in which the canonical join representation of $\1$ does not exist.  

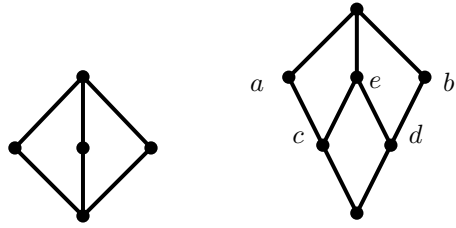
\begin{figure}[h]
  \centering
  \begin{picture}(-50,50)
  \put(95,51){$e$}
  \put(50,50){$a$}
  \put(123,50){$b$}
  \put(110,30){$d$}
  \put(66,30){$c$}
  \end{picture}
   \scalebox{1}{\scalebox{1}{ \includegraphics{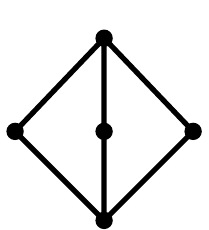}\qquad \qquad} \includegraphics{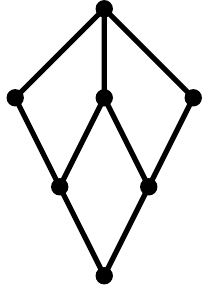}   }
   \vspace{5pt}
     \caption{Two finite lattices whose top elements have no canonical join representation.}
        \label{fig:meet_semi}
\end{figure}

In the modular lattice on the left each pair of atoms is a lowest-possible, irredundant join-representation for the top element. 
Since there is no \textit{unique} such join-representation, the canonical join representation for $\1$ does not exist.
Arguing dually, we see that the canonical meet representation for the bottom element $\0$ does not exist either.
In the lattice on right, each element has a canonical meet representation.
However, both $a\join d$ and $b\join c$ are minimal elements of $\ijr(\1)$.
Again, the canonical join representation of $\1$ does not exist.

In the lattice on the right, we observe the following failure of the distributive law:
both $e\join a$ and $e\join b$ are equal to the top element, but $e\join (a\meet b)$ is equal to $e$.
(A similar failure is easily verified among the atoms of the modular lattice.)
We will see that correcting for precisely this kind of failure of distributivity guarantees the existence of canonical join representations when $L$ is finite.

A lattice $L$ is \newword{join-semidistributive} if $L$ satisfies the following implication for every $x,y$ and $z$:
\begin{equation}\label{jsd}\tag{$SD_{\join}$}
\text{If $x\join y = x\join z$, then $x\join(y\meet z) = x\join y$}
\end{equation}
$L$ is \newword{meet-semidistributive} if it satisfies the dual condition:
\begin{equation}\label{msd}\tag{$SD_{\meet}$}
\text{If $x\meet y = x\meet z$, then $x\meet(y\join z) = x\meet y$}
\end{equation}
A lattice is \newword{semidistributive} if it is join-semidistributive and meet-semidistributive.
The following result, the finite case of \cite[Theorem~2.24]{free lattice}, says that this definition is equivalent to one given in the introduction.

\begin{theorem}\label{join_semi_cjr}
Suppose that $L$ is a finite lattice.
Then $L$ satisfies \ref{jsd} if and only if each element in $L$ has a canonical join representation.
Dually, $L$ satisfies \ref{msd} if and only if each element in $L$ has a canonical meet representation.
\end{theorem}

Assume that $L$ is a finite join-semidistributive lattice, and let $j\in \Irr(L)$.
We write $j_*$ for the unique element covered by $j$, and $\K(j)$ for the set of elements $a\in L$ such that $a\ge j_*$ and $a \not \ge j$.
 When it exists, we write $\kappa(j)$ for the unique maximal element of $\K(j)$.
 It is immediate that $\kappa(j)$ is meet-irreducible.
Below, we quote \cite[Theorem~2.56]{free lattice}:
\begin{proposition}\label{kappa}
A finite lattice is meet-semidistributive if and only if $\kappa(j)$ exists for each join-irreducible element $j$.
\end{proposition}

Suppose that $w\in L$.
For each $y\in \covdown(w)$, there is some element $j\in \can(w)$ such that $y\join j =w$ (because there is some element $j\in \can(w)$ such that $j\not \le y$).
For this $j$, the set $\can(w)$ join-refines $\{j, y\}$.
Because $\can(w)$ is an antichain, each $j'\in \can(w)\setminus \{j\}$ satisfies $j'\le y$.
Therefore, $j$ is the unique canonical joinand of $w$ such that $y\join j = w$.
We define a map $\eta: \covdown(w)\to \can(w)$ which sends $y$ to the unique canonical joinand $j$ such that $y\join j = w$.
\begin{lemma}\label{canonical cover}
Suppose that $L$ is a finite join-semidistributive lattice, and $w\in L$.
Then the map $\eta: \covdown(w) \to \can(w)$ is a bijection such that $y\ge \Join \can(w)\setminus\{\eta(y)\}$ and $y\in \K(\eta(y))$ for each $y\in \covdown(w)$.
\end{lemma}
\begin{proof}
Suppose there exist distinct $y$ and $y'$ in $\covdown(w)$ satisfying $\eta(y) = \eta(y')$.
Then, $y\join y' = w$, and $\can(w)$ does not join-refine $\{y, y'\}$ (because $\eta(y)$ is below neither $y$ nor $y'$).
We have a contradiction, because $\can(w)$ is the unique minimal element (in join-refinement) of $\ijr(w)$.
By this contradiction, we conclude that $\eta$ is injective.
Suppose that $j\in \can(w)$.
Since $\Join \can(w)$ is irredundant, $\Join (\can(w) \setminus \{j\}) < w$.
Thus, there is some $y\in \covdown(w)$ such that $y\ge \Join (\can(w) \setminus \{j\})$.
If $y \ge j$ then $y = w$, and that is absurd.
We conclude that $j = \eta(y)$, and that $\eta$ is a bijection.

We have already argued, in the paragraph above the statement of the proposition, that $y\ge \Join \can(w)\setminus \{\eta(y)\}$.
To complete the proof, suppose that $y \join \eta(y)_* = w$.
Since, $\can(w)$ does not join-refine $\{y ,\eta(y)_*\}$ (because $\eta(y)\not \le \eta(y)_*$ and $\eta(y)\not \le y$), we obtain a contradiction as above.
We conclude that $y \join \eta(y)_* <w$.
Since $y$ is covered by $w$, we have $y \join \eta(y)_* = y$.
Thus, $y \in \K(\eta(y))$, for each $y\in \covdown(w)$.
\end{proof}

As a consequence we obtain a proof of Proposition~\ref{canonical joinands first time}, which we restate here with the notation from of Lemma~\ref{canonical cover}.
\begin{proposition}\label{canonical joinands}
Suppose that $L$ is a finite join-semidistributive lattice, and $y$ is covered by $w$ in $L$.
Then, $\eta(y)$ is the unique minimal element of $L$ such that $\eta(y) \join y = w$.
\end{proposition}
\begin{proof}
Suppose that $x\in L$ has $x\join y = w$.
Since $\can(w)$ join-refines $\{x, y\}$ and $\eta(y)$ and $y$ are incomparable, we conclude that $\eta(y) \le x$.
\end{proof}

In fact, the previous proposition characterizes of finite join-semidistributive lattices.
(Similar constructions exist; for example, see the proof of \cite[Theorem~3-1.4]{sta2}.)
Because the proof is similar to the proof of Lemma~\ref{canonical cover}, we leave the details to the reader.
\begin{proposition}\label{converse}
Suppose that $L$ is a finite lattice.
The following conditions are equivalent:
\begin{enumerate}
\item\label{labeling} For each $w\in L$, there is a unique minimal element $\eta(y)\in L$ satisfying $y\join \eta(y) =w$, for each $y\in \covdown(w)$.
\item $L$ is join-semidistributive.
\end{enumerate}
\end{proposition}

Suppose that $L$ is a finite join-semidistributive lattice, $j\in \Irr(L)$ and $F$ is a face of the canonical join complex for $L$.
The following will be useful for determining when $F\cup \{j\}$ is also a face.

\begin{lemma}\label{lem: kappa canonical}
Suppose that $L$ is a finite join-semidistributive lattice and $j\in \Irr(L)$.
Then $j$ is a canonical joinand of $y\join j$, for each $y\in \K(j)$.
In particular, $j$ is a canonical joinand of $\Join F\join j$ if and only if $\Join F \join j > \Join F \join j_*$, for each subset $F$ of $L\setminus \{j\}$.
\end{lemma}
\begin{proof}
If $y= j_*$, then the first statement is obvious (because $\{j\}$ is the canonical join representation), so we assume that $y$ and $j$ are incomparable.
We write $w$ for the join $j\join y$, and we write $A=\{j'\in \can(w): j'\le j\}$ and $A'=\{j'\in \can(w): j'\le y\}$.
Because $\can(w)$ join-refines $\{j,y\}$, we have $A\cup A'=\can(w)$. 
Also, the set $A$ is not empty because the join $y\join j$ is irredundant.
We want to show that $A=\{j\}$.
Since $j$ is join-irreducible, it is enough to show that $j= \Join A$.
Since $y\ge \Join A'$, we see that $\Join A \join y = j\join y$.
If $\Join A < j$, then $j_* \join y = j \join y$, and that is impossible because $y\in \K(j)$.
We conclude that $j$ is a canonical joinand of $y\join j$.

If $\Join F \join j > \Join F \join j_*$, then $\Join F \join j_*\in \K(j)$.
We conclude that $j$ is a canonical joinand of $\Join F \join j$.
The remaining direction of the second statement is straightforward to verify.
\end{proof}

Specifically, Lemma~\ref{lem: kappa canonical} implies that if $F$ is a face of the canonical join complex then $A\cup \{j\}$ is a face if and only if $\Join A \join j > \Join A \join j_*$.

We close this subsection by quoting the following easy proposition (for example see \cite[Proposition~2.2]{arcs}), which says that the canonical join complex is indeed a simplicial complex.
\begin{proposition}\label{simplicial complex}
Suppose $L$ is a finite lattice, and the join $\Join A$ is a canonical join representation in $L$. 
Then each proper subset of $A$ also joins canonically.
 \end{proposition}

\subsection{The flag property}\label{semidistributive sec}

In this section we prove Theorem~\ref{flag}.
We begin by presenting the key arguments in one direction the proof: 
If $L$ is a finite semidistributive lattice, then its canonical join complex is flag.
Most of the work is done in the following two lemmas.

\begin{lemma}\label{simplex_incomparable}
Suppose that $L$ is a finite semidistributive lattice, and $F$ is a subset of $\Irr(L)$ such that $|F|\ge3$ and each proper subset of $F$ is a face in the canonical join complex for $L$.
Then the joins $\Join (F\setminus \{j\})$ and $\Join (F\setminus \{j'\})$ are incomparable for each distinct $j$ and $j'$ in $F$.
\end{lemma}
\begin{proof}
Without loss of generality we assume that $\Join F = \1$.
Suppose there exists distinct $j, j'\in F$ such that $\Join (F \setminus \{j\}) \ge \Join (F \setminus \{j'\})$.
On the one hand, we have $\Join (F\setminus \{j\}) \join \Join (F \setminus \{j'\}) = \Join F = \1$.
On the other hand, $\Join (F \setminus \{j\}) \join \Join (F \setminus \{j'\})$ is equal to $\Join (F \setminus \{j\})$.
Thus, $\Join(F\setminus \{j\}) = \1$.
Since $F$ has at least three elements, there exists $j'' \in F\setminus \{j, j'\}$.
We write $w'$ for $\Join (F \setminus \{j'\})$ and $w''$ for $\Join (F \setminus \{j''\})$.
Because both $F \setminus \{j'\}$ and $F \setminus \{j''\}$ are faces in the canonical join complex, $j$ is a canonical joinand for both $w'$ and $w''$.
Lemma~\ref{canonical cover} implies that there exists $y' \in \covdown(w')$ and $y'' \in \covdown(w'')$ such that $y', y''\in \K(j)$.
Moreover, $y' \ge \Join (F\setminus \{j, j'\})$ and similarly $y''\ge \Join (F\setminus \{j, j''\})$.
So, we have: \[y'\join y'' \ge \Join (F \setminus \{j, j'\})\join \Join (F \setminus \{j, j''\}) = \Join (F \setminus \{j\}).\]
Since $\Join (F \setminus \{j\}) =\1$ we conclude that $\Join \K(j) = \1$, contradicting Proposition~\ref{kappa}.
\end{proof}


\begin{lemma}\label{if direction}
Suppose that $L$ is a finite join-semidistributive lattice, and $F$ is a subset of $\Irr(L)$ satisfying the following conditions:
First, $|F|\ge3$; second, each proper subset of $F$ is a face in the canonical join complex for $L$; third, $\Join F$ is irredundant; fourth $F$ is not a face of the canonical join complex.
Then there exists $j\in F$ such that $\kappa(j)$ does not exist.
\end{lemma}
\begin{proof}
Without loss of generality, we assume that $\Join F = \1$.
Since the join $\Join F$ is irredundant, there exists some $j\in F$ such that $j\not \in \can(\1)$.
Lemma~\ref{lem: kappa canonical} implies that $j_* \join \Join (F\setminus \{j\}) = \1$.
Let $j'$ and $j''$ be distinct elements in $F\setminus \{j\}$.
As in the proof of Lemma~\ref{simplex_incomparable}, let $y'$ and $y''$ be the unique elements covered by $\Join F\setminus \{j'\}$ and $\Join F\setminus\{j''\}$, respectively, with $y'\join j = \Join F\setminus \{j'\}$ and $y'' \join j = \Join F\setminus\{j''\}$.
By Lemma~\ref{canonical cover}, $y'$ and $y''$ are both members of $\K(j)$, so that $y', y''\ge j_*$. 
Also $y' \ge \Join (F\setminus \{j, j'\})$ and $y''\ge~\Join(F \setminus \{j, j''\})$.
Therefore, $y'\join y'' \ge j_*\join  \Join (F\setminus \{j\}) =\1$.
The statement follows.
\end{proof}

\begin{proof}[Proof of one direction of Theorem~\ref{flag}.]
We show that if $L$ is semidistributive, then its canonical join complex is flag.
Suppose that $F\subset \Irr(L)$ such that $|F|\ge 3$ and each proper subset of $F$ is a face of the canonical join complex.
By Lemma~\ref{if direction}, it is enough to show that $\Join F$ is irredundant.
Without loss of generality, assume that $\Join F = \1$.
Lemma~\ref{simplex_incomparable} says that for each distinct $j$ and $j'$ in $F$, the joins $\Join (F\setminus \{j\})$ and $\Join (F\setminus \{j'\})$ are incomparable.
So, for any distinct $j$ and $j'$ in $F$, we have $\Join(F \setminus \{j\}) < \Join (F\setminus \{j\}) \join \Join (F\setminus \{j'\})$ which is $\1$.
We conclude that $\Join F$ is irredundant, and thus a face of the canonical join complex.
\end{proof}

We now turn to the other direction of Theorem~\ref{flag}.
In the following lemmas we will assume that $L$ is a finite join-semidistributive lattice which fails~\ref{msd}.
By Proposition~\ref{kappa}, there is some $j\in \Irr(L)$ such that $\kappa(j)$ does not exist.
Our goal is to construct a set $A\subset \Irr(L)$ satisfying:
\begin{enumerate}
\item\label{nonface} $A\cup \{j\}$ \textit{is not} a face in the canonical join complex for $L$ and
\item\label{edge face} each pair of elements in $A\cup \{j\}$ \textit{is} a face in the canonical join complex.
\end{enumerate}

The essential idea is that among all $Y\subset \Irr(L)$ satisfying~(\ref{nonface}), a set $A$ chosen as low as possible in $L$ will also satisfy~(\ref{edge face}).
For us, ``as low as possible in $L$'' means that $A$ is chosen to be minimal in join-refinement.
The argument is somewhat delicate because join-refinement is a preorder, not a partial order, on subsets of $L$.
So, we must take extra care to compare only antichains $Y\subset \Irr(L)$ satisfying~(\ref{nonface}).
To further emphasize this point, we write $A\refines B$ when $A$ join-refines $B$, for \textit{antichains} $A$ and $B$.
We write $\A_j$ for the collection of antichains $Y\subseteq L\setminus\{j\}$ satisfying $Y\cup\{j\}$ is an antichain.

\begin{lemma}\label{lem: kappa}
Suppose that $L$ is a finite join-semidistributive lattice and $j$ is in $\Irr(L)$ such that $\kappa(j)$ does not exist.
Let $X$ denote the set of $j'\in \Irr(L)\setminus \{j\}$ such that $j'\join j$ is a canonical join representation.
Then:
\begin{enumerate}
\item $\Join X \join j = \Join X \join j_*$;  
\item There exists a nonempty antichain $Y$ in $\A_j$ such that $\Join Y\join j=\Join Y\join j_*$.
\end{enumerate}
\end{lemma}
\begin{proof}
Assume that $\Join X\join j>  \Join X\join j_*$.
Lemma~\ref{lem: kappa canonical} says that $j$ is a canonical joinand of $\Join X\join j$.
Also, for each element $a$ in $\K(j)$, $j$ is a canonical joinand of $a\join j$.
That is, $a\join j$ has the canonical join representation $\Join X'\join j$ for some subset $X'\subset X$.
Thus $a\join j \le \Join X\join j$, and in particular $a\le \Join X\join j$.
Lemma~\ref{canonical cover} implies that there is a unique element $y\in \K(j)$ covered by $\Join X\join j$.
If $a$ is not less than $y$, then $y\join a = \Join X \join j$.
Proposition~\ref{canonical joinands} says that $j$ is the unique minimal element of $L$ whose join with $y$ is equal to $\Join X\join j$.
Therefore, $j\le a$, contradicting the fact that  $a\in \K(j)$.
We conclude that $a\le y$.
We have proved that $y=\kappa(j)$, contradicting our hypothesis.
Thus, $\Join X\join j = \Join X\join j_*$.

For the second statement, observe that if $X$ is empty, then Lemma~\ref{lem: kappa canonical} implies that $\K(j) = \{j_*\}$, contradicting the assumption that $\kappa(j)$ does not exist.
We conclude that $X$ is nonempty.
Since the antichain of maximal elements $Y\subseteq X$ satisfies $\Join Y = \Join X$, we have the desired result.
\end{proof}

Lemma~\ref{lem: kappa} implies that among all $Y$ in $\A_j$ satisfying $\Join Y \join j=\Join Y \join j_*$, there is a nonempty minimal (in join-refinement) antichain.
In particular, there is an antichain that is minimal with this property among the antichains of the set $X = \{j'\in \Irr(L)\setminus \{j\} : j'\join j\text{ is a canonical join representation}\}$.
For this antichain $A$, we have that $A\cup\{j\}$ is \textit{not} a face of the canonical join complex, while $\{a,j\}$ \textit{is} a face, for each $a\in A$.
The next two lemmas are key in showing that $\{a, a'\}$ is a face in the canonical join complex for pair $a, a'\in A$.

Before we begin, we point out two easy observations about the join-refinement relation.
\begin{enumerate}[label={(JR\arabic*)}, ref={(JR\arabic*)}]
\item For any subsets $S$ and $T$ of $L$, if $S$ join-refines $T$ then each subset $S'\subseteq T$ also does. \label{JR1}
\item Suppose that $S\cup \{x\}$ and $T\cup\{x\}$ are antichains.
Then, $S\cup \{x\} \refines T\cup\{x\}$ if and only if $S\refines T$.\label{JR2} 
\end{enumerate}
\begin{lemma}\label{lem: refinement-minimal}
Suppose that $L$ is a finite join-semidistributive lattice, and $j$ is in $\Irr(L)$ such that $\kappa(j)$ does not exist.
Among all nonempty antichains $Y$ in $\A_j$ such that $\Join Y \join j =\Join Y \join j_*$, let $B$ be minimal in join-refinement.
Then the join $\Join\left( B\setminus \{b\}\right)\join j$ is a canonical join representation, for each $b\in B$.
\end{lemma}
\begin{proof}
Lemma~\ref{lem: kappa} implies that such an antichain $B$ exists.
Observe that $B\setminus \{b\} \refines B$, by~\ref{JR1}.
We conclude that $\Join\left( B\setminus \{b\}\right)\join j_*<\Join \left(B\setminus \{b\}\right)\join j$.
Lemma~\ref{lem: kappa canonical} says that $j$ is a canonical joinand of $\Join\left( B\setminus \{b\}\right)\join j$ and not a canonical joinand of $\Join B\join j$.
Thus, 
\begin{equation}\label{relation}
\Join (B\setminus \{b\})\join j<\Join B\join j.
\end{equation}
Let $C\cup j$ be the canonical join representation of $\Join\left( B\setminus \{b\}\right)\join j$.
If we have $\Join\left( C\cup \{b\}\right)\join j_*<\Join \left(C\cup \{b\}\right)\join j$ then Lemma~\ref{lem: kappa canonical} says that $j$ is a canonical joinand of $\Join\left( C\cup \{b\}\right)\join j= \Join B\join j$.
That is a contradiction.
Therefore, $\Join \left(C\cup\{b\}\right)\join j =\Join \left(C\cup \{b\}\right)\join j_*$.

We claim that $C\cup \{b\} = B$.
Since $C\cup \{j\}$ is the canonical join representation for $\Join (B\setminus \{b\})\join j$, we have  $C\cup \{j\}\refines (B\cup\{j\}) \setminus \{b\}$.
By~\ref{JR2}, we have $C\refines B \setminus \{b\}$.
If $C\cup \{b\}$ is an antichain, then applying~\ref{JR2} again, we get $C\cup \{b\} \refines B$.
By minimality of $B$ we conclude that $C\cup \{b\} = B$, as desired.

So, we assume that $C\cup\{b\}$ is not an antichain.
By~\eqref{relation} we have \[\Join C \join j = \Join (B \setminus \{b\})\join j < \Join B \join j.\]
Therefore, there exists no $c\in C$ with $b\le c$.
Let $C'$ be the set of all $c\in C$ with $c< b$.
We make three easy observations:
First, $(C\setminus C') \cup \{b\}$ is member of $\A_j$.
Second, applying~\ref{JR1} to the relation $C\refines B\setminus \{b\}$, we have that $C\setminus C' \refines B\setminus \{b\}$.
By~\ref{JR2}, we conclude that $(C\setminus C')\cup \{b\} \refines B$.
Third, we have: \[\Join ((C\setminus C')\cup\{b\}) \join j=\Join (C\cup \{b\}) \join j=\Join (C\cup\{b\}) \join j_* = \Join ((C\setminus C')\cup\{b\}) \join j_*.\]
Therefore, by the minimality of $B$, we have $B=(C\setminus C' )\cup\{b\}$.
Since $C\refines B\setminus \{b\}$, we have that $C$ join-refines its proper subset $C\setminus C'$.
That is a contradiction (because $C$ is an antichain).
Thus, $C'$ is empty.
We have proved the desired result.
\end{proof}

\begin{lemma}\label{lem: minimal means canonical}
Suppose that $L$ is a finite join-semidistributive lattice and $j\in \Irr(L)$ such that $\kappa(j)$ does not exist.
Let $X$ be the set of $j'\in \Irr(L)\setminus \{j\}$ such that $j'\join j$ is a canonical join representation.
Let $A$ be nonempty and minimal in join-refinement among all antichains $Y\subseteq X$ such that $\Join Y\join j = \Join Y\join j_*$.
Then $A$ is minimal among all elements in $\A_j$, in join refinement,  with this property.
\end{lemma}
\begin{proof}
Lemma~\ref{lem: kappa} implies that such an antichain $A$ exists.
Suppose that $B\in \A_j$ satisfies $\Join B\join j = \Join B \join j_*$, and $B\refines A$.
Without loss of generality, assume that $B$ is minimal in join-refinement with this property.
If $B$ has two or more elements, then Lemma~\ref{lem: refinement-minimal} implies that $B\subset X$.
Therefore, $B=A$.
Thus we can assume that $B=\{b\}$.
Since $B$ join-refines $A$, there is some $a\in A$ such that $b\le a$.
Write $w$ for the element $a\join j$.
Since $a\join j$ is the canonical join representation of $w$, Lemma~\ref{canonical cover} implies that $\covdown(w)$ has precisely two elements, $y$ and $y'$.
Let $\eta(y) =j$ and $\eta(y') =a$, so that $y\in \K(j)$ and $y\ge a$.
Thus, we have $b\le a\le y$.
On the one hand, $(b\join j) \join y = (b \join j_*) \join y =y$.
On the other hand, $b\join (j \join y) = b\join w = w$.
By this contradiction, we have proved the result.
\end{proof}
As in the previous lemma, let $A$ be minimal (in join-refinement) among all of the antichains $Y\subseteq X$ with the property that $\Join Y \join j = \Join Y \join j_*$.
In the next lemma, we show that each pair $\{a, a'\}$ in $A$ is canonical join representation.

\begin{lemma}\label{edges}
Suppose that $L$ is a finite join-semidistributive lattice and $j\in \Irr(L)$ such that $\kappa(j)$ does not exist.
Let $X$ be the set of $j'\in \Irr(L)\setminus \{j\}$ such that $j'\join j$ is a canonical join representation.
Let $A$ be nonempty and minimal in join-refinement among all antichains $Y\subseteq X$ such that $\Join Y\join j = \Join Y\join j_*$.
Then each pair of elements in $A$ is a face in the canonical join complex.
\end{lemma}
\begin{proof}
Lemma~\ref{lem: minimal means canonical} says that $A$ is minimal (in join-refinement) in $A_j$ among all $B\in A_j$ with the property that $\Join B \join j = \Join B\join j_*$.
So, Lemma~\ref{lem: refinement-minimal} says that for each $a\in A$, the join $\Join\left( A\setminus \{a\}\right)\join j$ is a canonical join representation.
If $A$ has three or more elements, then each pair of elements joins canonically by Proposition~\ref{simplicial complex}.
Assume that $A$ has two elements, $a_1$ and $a_2$.
Minimality of $A$ (in join-refinement) implies that the join $a_1\join a_2$ is irredundant.
We will argue that $a_1$ is a canonical joinand of $a_1\join a_2$, and complete the proof by symmetry.

Assume that $(a_{1})_* \join a_2 = a_1\join a_2$.
We observe that $(a_{1})_* \join a_2 \join j = (a_{1})_* \join a_2 \join j_*$. 
Since $A\subseteq X$, we have that both $\{a_1, j\}$ and $\{a_2, j\}$ are faces in the canonical join complex.
If $(a_{1})_* < j$, then we have $a_2\join j = a_2\join j_*$, contradicting Lemma~\ref{lem: kappa canonical}.
Also, $j\not \le (a_{1})_*$ because $a_1$ is not comparable to $j$.
So, we have $\{(a_{1})_*, a_2\}\in A_j$ with $(a_{1})_* \join a_2 \join j = (a_{1})_* \join a_2 \join j_*$, and $\{(a_{1})_*, a_2\}$ join-refines $\{a_1, a_2\}$.
But this contradicts Lemma~\ref{lem: minimal means canonical} which says that $A$ is minimal in $\A_j$.
By this contradiction, we conclude that $(a_1)_* \join a_2 < a_1\join a_2$.
Lemma~\ref{lem: kappa canonical} says that $a_1$ is a canonical joinand of $a_1\join a_2$.
\end{proof}

Finally, we complete the proof of the main result.

\begin{proof}[Proof of the remaining direction of Theorem~\ref{flag}]
We show that if $L$ is a finite join-semidistributive lattice and the canonical join complex for $L$ is flag, then $L$ is semidistributive.
By Proposition~\ref{kappa}, it is enough to show that for each $j\in \Irr(L)$ the element $\kappa(j)$ exists.

Suppose $j\in \Irr(L)$ and $\kappa(j)$ does not exist. 
As above, let $X$ be the set of $j'\in \Irr(L)\setminus \{j\}$ such that $j'\join j$ is a canonical join representation. 
Among all nonempty antichains of $X$, choose $A$ to be minimal in join-refinement with the property that $\Join A \join j = \Join A\join j_*$.
Lemma~\ref{lem: kappa} implies that such an antichain $A$ exists.
Lemma~\ref{lem: kappa canonical} implies that $A\cup\{j\}$ is not face of the canonical join complex.
Since $A\subseteq X$, we have that $\{a,j\}$ is a face of the canonical join complex, for each $a\in A$.
In particular, $A$ has at least two elements.
Finally, Lemma~\ref{edges} says that $\{a, a'\}$ is face in the canonical join complex, for each pair $a, a'\in A$.
We have reached a contradiction to our hypothesis that the canonical join complex is flag.
By this contradiction, we conclude that $L$ is semidistributive.
\end{proof}

Suppose that $m$ is meet-irreducible and write $m_*$ for the unique element covering $m$.
When it exists, let $\kappa_*(m)$ be the unique smallest element $j\in L$ with $j\le m_*$ and $j\not \le m$.
It is immediate that $\kappa_*(m)$ is join-irreducible.
Proposition~\ref{kappa}, applied to the dual lattice, says that $L$ is meet-semidistributive if and only if $\kappa_*(m)$ exists for each meet-irreducible element $m$.
In fact, $L$ is semidistributive if and only if $\kappa$ is a bijection, with inverse map $\kappa_*$; this is the finite case of \cite[Corollary~2.55]{free lattice}.
Applying the dual argument for the canonical meet complex, we immediately obtain the following result.
(Recall that Theorem~\ref{join_semi_cjr} says that each element in $L$ has a canonical meet representation if and only if $L$ is meet-semidistributive.)
\begin{corollary}\label{cor: meet}
Suppose that $L$ is a finite meet-semidistributive lattice.
Then, the canonical meet complex for $L$ is flag if and only if $L$ is semidistributive.
\end{corollary}

Next, we prove Corollary~\ref{iso} by showing that the bijection $\kappa$ taking a join-irreducible element $j$ to $\kappa(j)$ induces an isomorphism from the canonical join complex of $L$ to the canonical meet complex of $L$.
\begin{proof}[Proof of Corollary~\ref{iso}]
Corollary~\ref{cor: meet} says that the canonical meet complex of $L$ is flag, so it is enough to show that $\kappa$ bijectively maps edges of the canonical join complex to edges of the canonical meet complex.
Suppose that $\{j_1, j_2\}$ is a face of the canonical join complex, and write $m_1$ for $\kappa(j_1)$ and $m_2$ for $\kappa(j_2)$.
Suppose that $m_1\meet m_2= (m_{1})_* \meet m_2$.
Lemma~\ref{canonical cover} implies that there exists some $y\in \covdown(j_1\join j_2)$ satisfying : $j_1 \le y\le \kappa(j_2)$ (see Figure~\ref{isomorphism helper} for an illustration).
Since $j_1 \le (m_{1})_*$, we conclude that $j_1 \le (m_{1})_* \meet m_2= m_1\meet m_2$.
We see that $j_1 \le m_1$ and that is a contradiction.
Therefore, $(m_{1})_*\meet m_2 > m_1 \meet m_2$.
By the dual statement of Lemma~\ref{lem: kappa canonical}, we conclude that $m_1$ is a canonical meetand of $m_1\meet m_2$, and by symmetry $m_2$ is also a canonical meetand of $m_1\meet m_2$.
The dual argument establishes the desired isomorphism. 
\end{proof}

\begin{figure}[h]
  \centering
\begin{tikzpicture}
\draw [gray, dashed] (1.4,-1.4) arc [radius=2, start angle=-45, end angle= 45];
\draw[gray, dashed] (0,1.4) arc [radius=2, start angle=135, end angle=225];
\draw [black, thick] (0.07,1.45) -- (.68,2.1);
\draw [gray, dashed] (1.43,1.4) -- (.7,2.1);

\draw[fill] (0.07,1.45) circle [radius=0.07];
\node [above] at (-.1,1.4) {$y$};

\draw[fill] (0.68,2.1) circle [radius=0.07];
\node [above] at (0.61,2.1) {$j_1\join j_2$};

\draw[fill] (.01,-1.4) circle [radius=0.07];
\node [above] at (0.1,-1.37) {$j_1$};
\draw [black, thick] (0.01,-1.4) -- (.01,-2);
\draw[fill] (.01,-2) circle [radius=0.07];
\node [below] at (0,-2.1) {$(j_1)_*$};

\draw [gray, dashed] (.01,-1.4) -- (-1.5,.05);

\draw[fill] (-1.5,.05) circle [radius=0.07];
\node [above] at (-1.5,.05) {$(m_1)_*$};

\draw[fill] (-1.5,-.6) circle [radius=0.07];
\node [below] at (-1.65,-.65) {$\kappa(j_1)=m_1$};
\draw [black, thick] (-1.5,.05) -- (-1.5,-.6);

\draw[fill] (1.4,-1.4) circle [radius=0.07];
\node [above] at (1.37,-1.37) {$j_2$};
\draw [black, thick] (1.4,-1.4) -- (1.4,-2);

\draw[fill] (1.4,-2) circle [radius=0.07];
\node [below] at (1.5,-2.1) {$(j_2)_*$};
\end{tikzpicture}
 \caption{The above figure is an illustration of the argument for the proof of Corollary~\ref{iso}. Dashed gray lines represent relations in $L$, while thick black lines represent cover relations.}
        \label{isomorphism helper}
\end{figure}
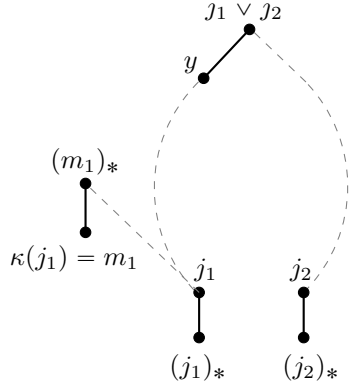 

We close this section by relating Corollary~\ref{iso} to Example~\ref{distributive lattices} and Example~\ref{noncrossing}, from Section~\ref{Examples}.

\begin{remark}\label{rowmotion}
\normalfont
Suppose that $F$ is a face of the canonical join complex for a finite semidistributive lattice $L$.
Corollary~\ref{iso} says that $\Join \kappa (F)$ is a canonical meet representation.
By taking the canonical join representation of $\Join \kappa (F)$, we can view the map $\kappa$ as an operation on the canonical join complex.
Similarly, we can view $\kappa_*$ as an action on the canonical meet complex.

The main premise of~\cite{kreweras} is that the action of Kreweras complementation on the set of noncrossing partitions and the action of Panyshev complementation on the set of nonnesting partitions (that is, the set of antichains in the root poset for a finite cystrallographic root system) coincide.
Indeed, both maps are an instance of the operation of $\kappa$ (or $\kappa_*$) on the canonical join complex (or canonical meet complex).

On the one hand, the action of $\kappa$ on the canonical join complex for the Tamari lattice coincides with Kreweras complementation (recall from Example~\ref{noncrossing} that canonical join representations in the Tamari lattice are essentially noncrossing partitions).
On the other hand, Panyshev complementation is a special case of an operation on the set of antichains in a finite poset $\poset$  called \newword{rowmotion}, as we now explain.
When $A$ is an antichain in $\poset$, we write $\Row(A)$ for the antichain $\{x\in \poset: x \text{ is minimal among elements not in $I_A$}\}$.
(Our notation is based on \cite{row}.
See also \cite{cyclic sieving, row reference brouwer, fon-der-flaass, row reference deza, fon-der-flaass II, row reference rush}.)
So, we have $I_A = I^{\Row(A)}$.
It follows immediately from the definition of $\kappa_*$ that $\kappa_*(I^y)\mapsto I_y$.
We obtain the following result.
\begin{proposition}\label{row prop}
Suppose that $\poset$ is a finite poset, and $A$ is an antichain in $\poset$.
Then the map $\kappa_*$, acting on faces of the canonical meet complex of $J(\poset)$, sends the order ideal $I^A$ to the order ideal $I^{\Row(A)}$.
\end{proposition}
\end{remark}

\subsection{Crosscut-simplicial lattices}

In this section, we prove Corollary~\ref{crosscut converse}.
Recall that one direction of the proof was given as \cite[Theorem~3.1]{crosscut}.
Because it is easy, we give an alternative argument below.
Write $A$ for the set of atoms in $L$.
When $L$ is a finite semidistributive lattice every join of two atoms is a canonical join representation.
In particular,  Theorem~\ref{flag} implies that each distinct subset of atoms gives rise to a distinct element in $L$.
Thus the crosscut complex for $L$ is either the boundary of the simplex on $A$ or equal to the simplex on $A$, depending on whether $\Join A = \1$ or $\Join A <\1$.
Since each interval in $L$ inherits semidistributivity, it follows that $L$ is crosscut-simplicial.

Before we proceed with the proof of the converse, we point out that the join-semidistributivity hypothesis in Corollary~\ref{crosscut converse} is crucial.
(For example, consider the crosscut-simplicial lattice shown in Figure~\ref{crosscut}.
This lattice fails both \ref{jsd} and~\ref{msd}.)
\begin{figure}[h]
  \centering
   \scalebox{1}{ \includegraphics{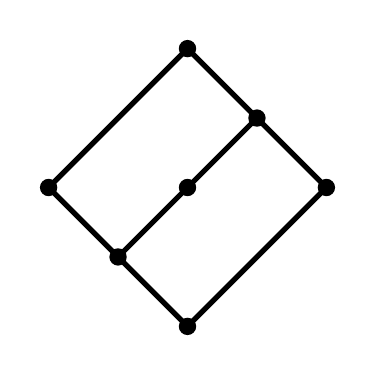}}
     \caption{A finite crosscut-simplicial lattice failing both \ref{jsd} and \ref{msd}.}
        \label{crosscut}
\end{figure}   
Join-semidistributivity gives us a powerful restriction: 
A finite join-semidistributive lattice $L$ fails \ref{msd} if and only if $L$ contains the lattice shown in Figure~\ref{fig:join_semi} as a sublattice (\cite[Theorem~5.56]{free lattice}).

We now begin our proof.
The following lemmas will be useful; the first is \cite[Lemma~9-2.5]{regions}.
\begin{lemma}\label{BEZ}
Suppose that $L$ is a finite lattice satisfying the following property:
If $x$, $y$, and $z$ are elements of $L$ with $x\meet y=x\meet z$ and if $y$ and $z$ cover a common element, then $x\meet(y\join z) = x\meet y$.
Then, $L$ is meet-semidistributive.
\end{lemma}

\begin{lemma}\label{crosscut lemma}
Suppose that $L$ is a finite join-semidistributive lattice that is not meet-semidistributive.
Then there exists $x$, $y$, and $z$ such that $y\join z >x$ and $x$, $y$, and $z$ cover a common element.
\end{lemma}
\begin{proof}
We prove the proposition by induction on the size of $L$.
As mentioned above, $L$ contains the lattice shown in Figure~\ref{fig:join_semi} as  sublattice, and this proves the base case.
By Lemma~\ref{BEZ}, we can assume that there exist $x$, $y$, and $z$ in $L$ such that $x\meet y= x\meet z$, $x\meet (y\join z) \ne x\meet y$, and $\covdown(y)\cap \covdown(z)$ is not empty.
 We choose such a triple so that the set $\{x,y,z\}$ is minimal in join-refinement, among all such triples.
 Write $a$ for the element in $\covdown(y)\cap \covdown(z)$ (if there is more than one element in $\covdown(y)\cap \covdown(z)$, then $y\meet z$ does not exist).
If $x$ also covers $a$, then we are done (because if $x\covers a$ and $y\join z\not > x$, then $(y\join z) \meet x = a$, and that contradicts our assumption that $\{x,y,z\}$ fail \ref{msd}).
So we assume that $x$ does not cover $a$.

We first prove that $x< y\join z$ (see Figure~\ref{crosscut incomparable} for an illustration).
We write $w$ for $x\meet(y\join z)$.
Since $x\meet y = x\meet z$, we have $x\meet y < w$ (because $x$, $y$ and $z$ fail~\ref{msd}, the inequality is strict).
On the one hand $w\meet(x\meet y) = x\meet y$.
On the other hand, $x\ge w$, so $(x\meet w)\meet y = w\meet y$.
By symmetry, $w\meet z= x\meet z$.
Therefore, $w\meet y= w\meet z$.
Observe that $w\ne y\meet w$ (otherwise $w\le y\meet x$, and that is absurd).
Since, $w\meet (y\join z) = w$ we have $\{w, y, z\}$ fails~\ref{msd}.
Since $\{w, y,z\}$ join-refines $\{x,y,z\}$, minimality of $\{x,y,z\}$ implies that $w=x$.
We have proved the claim that $y\join z > x$.
By induction, we may assume that $y\join z =\1$.

\begin{figure}[h]
  \centering
\begin{tikzpicture}
\draw [gray, dashed] (1.4,-1.4) arc [radius=2, start angle=-45, end angle= 45];
\draw[gray, dashed] (0,1.4) arc [radius=2, start angle=135, end angle=225];
\draw [gray, dashed] (1.43,1.4) -- (.7,2.1);
\draw [gray, dashed] (0,1.4) -- (.7,2.1);


\draw[fill] (0.68,2.1) circle [radius=0.07];
\node [above] at (0.68,2.1) {$y\join z$};

\draw[fill] (0.75,-2.25) circle [radius=0.07];
\node [below] at (0.75,-2.3) {$a$};

\draw [black, thick] (0.75,-2.3) -- (1.4,-1.4);
\draw [black, thick] (0.75,-2.3) -- (.01,-1.4);

\draw[fill] (0,-1.4) circle [radius=0.07];
\node [left] at (0,-1.5) {$y$};

\draw[fill] (1.4,-1.4) circle [radius=0.07];
\node [right] at (1.4,-1.5) {$z$};

\draw[fill] (0.72,-0.3) circle [radius=0.07];
\node [left] at (0.73,-0.34) {$w$};
\draw [gray, dashed] (0.73,-.17) -- (.69,2);
\draw [gray, dashed] (0.73,-0.45) -- (.75,-2.1);

\draw[fill] (-.9,1.55) circle [radius=0.07];
\node [left] at (-1,1.55) {$x$};
\draw [gray, dashed] (-.85,1.4) arc [radius= 2.6, start angle=195, end angle=242];
\end{tikzpicture}
 \caption{The above figure is an illustration of the argument for the proof that $y\join z > x$. Dashed gray lines represent relations in $L$, while thick black lines represent cover relations.}
        \label{crosscut incomparable}
\end{figure}

Next, we claim that $x\join y$ and $x\join z$ are incomparable.
By way of contradiction assume that  $x\join z \ge x\join y$, so we have $x\join z\ge x, y, z$.
Therefore, $z\join x = z\join y$.
Since $z\join (x\meet y) =z$ and $L$ is join-semidistributive, we have $z=\1$, contradicting the fact that $x\meet z \ne x\meet( y\join z)$.
We have proved the claim that $x\join y$ and $x\join z$ are incomparable.

Suppose that $\{y,z\} = \covup(a)$.
Then, either $y\le a\join x$ or $z\le a\join x$, but not both.
(Indeed, if $x\join a \ge y, z$ then $x\join a = \1$, so $x\join a = x\join y =x\join z$.
This contradicts the fact that $x\join y$ and $x\join z$ are incomparable.)
If $y< x\join a$ then $y<x\join a\le x\join z$.
Thus we have $x\join y \le x\join z$, contradicting the fact that $x\join y$ and $x\join z$ are incomparable.
We conclude that there is some $w'\in \covup(a) \setminus \{y,z\}$ with $w'\le a\join x$.
The triple $\{w', y, z\}$ satisfies the statement of the proposition.
\end{proof}

\begin{proof}[Proof of Theorem~\ref{crosscut converse}]
We prove that if $L$ is join-semidistributive and crosscut-simplicial then it is semidistributive.
Assume that $L$ is fails~\ref{msd}.
Lemma~\ref{crosscut lemma} says that there exists $x$, $y$ and $z$ covering a common element $a\in L$ such that $y\join z> x$.
In particular, the interval $[a, y\join z]$ is not crosscut-simplicial because $\{y,z\}$ is not a face in the crosscut complex.
That is a contradiction.
Therefore, $L$ is a finite semidistributive lattice, and the statement follows from Theorem~\ref{flag}.
\end{proof}

\section{Lattice-theoretic constructions}\label{operations}
\subsection{Sublattices and quotient lattices}\label{subs}
A map $\phi:L\to L'$ between lattices $L$ and $L'$ is a \newword{lattice homomorphism} if $\phi$ respects the meet and join operations.
The image of $\phi$ is a \newword{sublattice} of $L'$ and a \newword{lattice quotient} of $L$.
It is immediate that each sublattice of a semidistributive lattice is also semidistributive.
When $L$ is finite, the image $\phi(L)$ also inherits semidistributivity (see  \cite[Proposition~1-5.24]{regions}).
(Outside of the finite case, it is not generally true that if $L$ is semidistributive, then $\phi(L)$ is semidistributive; similarly for meet and join-semidistributivity.)
We obtain the following result as an immediate corollary of Theorem~\ref{flag}.
\begin{corollary}\label{subs and facs}
Suppose that $L$ is a finite join-semidistributive lattice whose canonical join complex is flag.
Then, the canonical join complex for each sublattice and quotient lattice of $L$ is also flag.
\end{corollary}


An equivalence relation $\Theta$ on $L$ is a \newword{lattice congruence} if $\Theta$ satisfies the following: if $x\equiv_{\Theta} y$, then $x\join t \equiv_{\Theta} y\join t$ and $x\meet t\equiv_{\Theta} y\meet t$ for each $x,y,$ and $t$ in $L$ (see \cite[Lemma~8]{foundations}).
It is immediate that the fibers of a lattice homomorphism $\phi$ constitute a lattice congruence of $L$.
Conversely, each lattice congruence also gives rise to a lattice quotient (see \cite[Theorem~11]{foundations}).

When $L$ is finite, $\Theta$ is lattice congruence if and only if it satisfies the following:
Each class is an interval; the map $\pidown^{\Theta}$ sending $x\in L$ to the smallest element in its $\Theta$-class is order preserving; the map $\piup_{\Theta}$ sending $x\in L$ to the largest element in its $\Theta$-class is order preserving.
Both $\pidown^{\Theta}$ and $\piup_{\Theta}$ are lattice homomorphisms onto their images such that $\pidown^{\Theta}(L)$ and $\piup_{\Theta}(L)$ are isomorphic lattice quotients of $L$.
The lattice quotient $\pidown^{\Theta}(L)$ is a sub-join-semilattice of $L$,  but not generally a sublattice of $L$.
Similarly, $\piup_{\Theta}(L)$ is a sub-meet-semilattice of $L$.

Below we quote \cite[Proposition~6.3]{shardint}.
In the proposition, a join-irreducible element $j\in L$ is \newword{contracted} by the congruence $\Theta$ if $j$ is congruent to the unique element that it covers.
\begin{proposition}\label{facs cjr}
Suppose that $L$ is a finite join-semidistributive lattice and $\Theta$ is a lattice congruence on $L$ with associated projection map $\pidown^{\Theta}$.
Then, the element $x$ belongs to $\pidown^{\Theta}(L)$ if and only if no canonical joinand of $x$ is contracted by~$\Theta$.
\end{proposition}
Suppose that $x\in \pidown^{\Theta}(L)$.
Since $\pidown^{\Theta}$ is a sub-join-semilattice of $L$, the canonical join representation of $x$ taken in the lattice quotient $\pidown^{\Theta}(L)$ is equal to the canonical join representation taken in $L$.

\begin{corollary}\label{facs cmplx}
Suppose that $L$ is a finite join-semidistributive lattice with canonical join complex $\Delta$, and $\Theta$ is a lattice congruence of $L$.
Then, the canonical join complex of $\pidown^{\Theta}(L)$ is the induced subcomplex of $\Delta$ supported on the set of join-irreducible elements not contracted by $\Theta$.
\end{corollary}
\begin{remark}\label{sublattices}
\normalfont
The canonical join complex of a sublattice $L'$ of $L$ need not be an induced subcomplex of $\Delta$.
In fact, the sets $\Irr(L')$ and $\Irr(L)$ may be disjoint. 
For example, consider the canonical join complex of the sublattice $\{\0,\1\}$ in the boolean lattice $B_n$, where $n>1$.
\end{remark}

\begin{remark}\label{forcing}
\normalfont
In general, not every induced subcomplex of $\Delta$ is the canonical join complex for a lattice quotient of $L$.
Each lattice congruence is determined by the set of join-irreducible elements that it contracts.
But, a given collection of join-irreducible elements may not correspond to a lattice congruence.
For $j$ and $j'$ in $\Irr(L)$, we say that $j$ \newword{forces} $j'$ if every congruence that contracts $j$ also contracts $j'$.
In $N_5$ pictured in Figure~\ref{pentagon} both $a$ and $b$ force $c$.
So, for example, there is no quotient of $N_5$ whose canonical join complex is the subcomplex induced by $\{b, c\}$.
\end{remark}

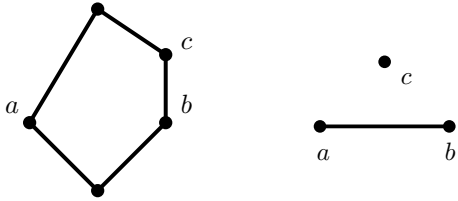
\begin{figure}[h]
  \centering
   \scalebox{1}{ \includegraphics{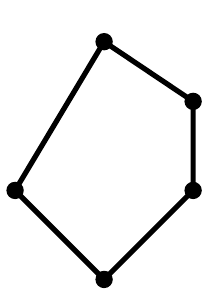}}
   \begin{picture}(0,50)(0,50)
   \put(-2,82){$b$}
   \put(-2,106){$c$}
   \put(-68.5,82){$a$}
   \end{picture}
   \qquad \qquad
      \scalebox{.95}{ \includegraphics{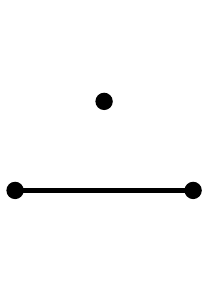}
   \begin{picture}(0,50)(0,50)
   \put(-9.5,65){$b$}
   \put(-27,95){$c$}
   \put(-60.5,65){$a$}
   \end{picture}}
     \caption{The pentagon lattice $N_5$ and its canonical join complex.}
        \label{pentagon}
\end{figure}

\subsection{Products and sums}
In the following easy propositions, we construct new semidistributive lattices from old ones, and give the corresponding construction for the canonical join complex.
\begin{proposition}\label{join}
Suppose that $L_1$ and $L_2$ are finite, join-semidistributive lattices with corresponding canonical join complex $\Delta_i$ for $i=1,2$.
Then the canonical complex for $L_1\times L_2$ is the join $\Delta_1\ast \Delta_2$.
\end{proposition}

The \newword{ordinal sum} of lattices $L_1$ and $L_2$ written $L_1\oplus L_2$ is the lattice whose set of elements is the disjoint union $L_1 \uplus L_2$, ordered as follows:
$x\le y$ if and only if $x\le y$ in $L_i$, for $i=1,2$, or $x\in L_1$ and $y\in L_2$.

\begin{proposition}\label{ordinal sum}
Suppose that $L_1$ and $L_2$ are finite, join-semidistributive lattices with corresponding canonical join complex $\Delta_i$, for $i=1,2$.
Then the canonical join complex of $L_1 \oplus L_2$ is equal to the disjoint union $\Delta_1 \uplus \Delta_2 \uplus \{v\}$, in which the vertex $v$ corresponds to the minimal element of $L_2$.
\end{proposition}

We define the \newword{wedge sum} $L_1 \wedgesum L_2$ to be the lattice quotient of the ordinal sum $L_1 \oplus L_2$ in which the minimal element of $L_2$ is identified with the maximal element of $L_1$.
(Our nonstandard terminology is inspired by the wedge sum of topological spaces.)

\begin{proposition}\label{wedge sum}
Suppose that $L_1$ and $L_2$ are finite, join-semidistributive lattices with corresponding canonical join complex $\Delta_i$, for $i=1,2$.
Then the canonical join complex of $L_1 \wedgesum L_2$ is equal to the disjoint union $\Delta_1 \uplus \Delta_2$.
\end{proposition}

\subsection{Day's doubling construction}
A subset $C$ of $L$ is \newword{order-convex} if for each $x, y\in C$ with $x\le y$, we have that the interval $(x,y)$ belongs to $C$.
Suppose that $C\subseteq L$ is order convex, and let $\2$ be the two element chain $0<1$.
We write $X$ for the set of elements $x\in L$ such that $x\ge c$ for some $c\in C$.
Define $L[C]$ to be the following induced subposet of $L\times \2$:
\[[((L\setminus X)\cup C)\times 0] \uplus (X\times 1)\]
We say that $L[C]$ is obtained by \newword{doubling} $L$ with respect to $C$.
This procedure, due to Day \cite{day2}, is defined more generally for all posets.
If $L$ is a lattice, then $L[C]$ is a lattice and the map $\pi_C: L[C] \to L$ given by $(x,\epsilon) \mapsto x$ is a surjective lattice homomorphism (see \cite{day2} or \cite[Lemma~6.1]{crosscut}).
In the next proposition, we show that when $C$ is an interval in $L$, doubling $L$ with respect to $C$ also preserves semidistributivity.

\begin{proposition}\label{doubling}
Suppose that $L$ is a finite semidistributive lattice, $I=[a,b]$ is an interval in $L$, and write $\E$ for the edge set of the canonical join graph for $L$.
Then $L[I]$ is semidistributive, and the canonical join graph for $L[I]$ has edge set \[ \E' \uplus \left\{ \{(j,0),(a,1)\}: j \in \can(w) \text{ for } w\in I\text{ and } j \not\le a\right\},\]
where $\E'$ is the set of pairs $\{(j,\epsilon), (j',\epsilon')\}$, such that $\{j,j'\}\in \E$, and $(j,\epsilon)$ and $(j',\epsilon')$ are the minimal elements of the fibers $\pi_{I}^{-1}(j)$ and $\pi_{I}^{-1}(j')$, respectively.

\end{proposition}
In the proof below we check that $L[I]$ satisfies~(\ref{labeling}) from Proposition~\ref{converse} (and the obvious dual argument gives meet-semidistributivity).
One can also verify semidistributivity directly for $L[I]$ using \cite[Lemma~6.1]{crosscut}.
Our approach has the advantage of giving the canonical join representation of each element of $L[I]$.
In either case, the argument is tedious but, at least, elementary. 

\begin{proof}
Suppose that $(w,\epsilon)$ is not in $I\times 1$, where $\epsilon = 0,1$.
Observe that the map $\pi_I: (y,\epsilon')\mapsto y$ is a bijection from $\covdown((w,\epsilon))$ to $\covdown(w)$.
For each $y\in \covdown(w)$, write $\eta(y)$ for the unique minimal element of $L$ satisfying $y\join \eta(y)=w$, and $(y,\epsilon')$ for the corresponding element in $\covdown((w,\epsilon))$.
Let $(\eta(y), \epsilon'')$ be the minimal element of the fiber $\pi_{I}^{-1}(\eta(y))$ in $L[I]$.
We claim  that $(\eta(y),\epsilon'') \join (y,\epsilon') = (w, \epsilon)$.
If $\epsilon=0$, the claim is immediate, and if $\epsilon=1$ then the claim follows from the fact that $(w,0)\not \in L[I]$.
It is straightforward, using the surjection $\pi_{I}$, to check that $(\eta(x), \epsilon'')$ is the unique minimal element of $L[C]$ whose join with $(x,\epsilon')$ is equal to $(y, \epsilon)$.


Suppose that $(w,1)\in I\times 1$.
If $w=a$, it is immediate that $(w,1)$ satisfies condition~(\ref{labeling}) of Proposition~\ref{converse}.
So we assume that $w>a$.
Observe that the lower covers of $(w,1)$ are $(y,1)$ such that $y\in \covdown(y)\cap I$ and $(w,0)$.
For each $y\in \covdown(w)\cap I$, we claim that the set $\{\eta(y): y\in \covdown(w)\cap I\}$ is precisely the set of canonical joinands of $w$ that are not weakly below $a$.
If $y\in \covdown(w)\setminus I$, then $y\join a = w$.
By minimality of $\eta(y)$, we conclude that $\eta(y)\le a$.
If $y\in \covdown(w)\cap I$ and $\eta(y)\le a$, then $\eta(y)\join y = y$, which is a contradiction.
The claim follows.
As above, it is straightforward to check that $(\eta(y), 0)$ is the unique minimal element in $L[I]$ whose join with $(y,1)$ is equal to $(w,1)$, for each $y\in \covdown(w)\cap I$.

Suppose that $(w',\epsilon')\join (w,0)= (w,1)$, where $\epsilon'\in \{0,1\}$.
Then $\epsilon'=1$, and we have $w'\ge a$.
Therefore, $(a, 1)$ is the unique minimal element whose join with $(w,0)$ is equal to $(w,1)$.
Proposition~\ref{converse} says that $L$ is join-semidistributive.
The second statement follows from Proposition~\ref{canonical joinands}.
\end{proof}


Below we gather some useful facts that follow immediately from the proof of Proposition~\ref{doubling}.
\begin{proposition}\label{consequences}
Suppose that $L$ is a finite semidistributive lattice, $I =[a,b]$ is an interval in $L$, and $j\in \Irr(L)$ such that $j\ne a$.
For each $w\in L$ and $\epsilon, \epsilon'\in \{0,1\}$ the following statements hold:
\begin{enumerate}
\item If $(j,\epsilon)$ is a canonical joinand of $(w,\epsilon')$ in $L[I]$, then $j$ is a canonical joinand of $w$.
\item If $(j,\epsilon)$ is a canonical joinand of $(w,\epsilon')\in I\times \2$ then $\epsilon=0$.
\item If $(j,\epsilon)$ is a canonical joinand of $(w,0)\in I\times 0$ and $j\not\le a$, then $(j,\epsilon)$ is also a canonical joinand of $(w,1)$.
\item $(w,\epsilon')$ has $(a,1)$ as canonical joinand if and only if $(w,\epsilon')\in I\times 1$.
\end{enumerate}
\end{proposition}

A lattice is \newword{congruence uniform} if it is obtained from the one element lattice by a finite sequence of doublings of intervals.
Suppose that $L$ is a finite congruence uniform lattice.
Proposition~\ref{doubling} says that after each iteration of the doubling procedure, the resulting lattice has exactly one additional join-irreducible element, namely $(a,1)$, where $a$ is the smallest element of the interval that is doubled.
Thus the canonical join graph of each congruence uniform lattice $L$ has a natural labeling, in which the vertex labeled $i$ is the join-irreducible element that is added in the $i^{\th}$ step of the doubling sequence for $L$.

\begin{remark}\label{warning!}
\normalfont
Non-isomorphic congruence uniform lattices may have the same labeled canonical join graphs.
For example, doubling the boolean lattice $B_2$ with respect to any singleton interval $I = \{x\}$, for $x\in B_2$, results in the labeled canonical join graph depicted in Figure~\ref{join graph} below.
When $x$ is equal to $\0$ or $\1$, we obtain the ordinal sums $B_0\oplus B_2$ and $B_2 \oplus B_0$, respectively.
When $x$ is either join-irreducible element of $B_2$, the resulting lattice is isomorphic to $N_5$ from Figure~\ref{pentagon}.
\end{remark}
\begin{figure}[h]
  \centering
   \scalebox{.75}{ \includegraphics{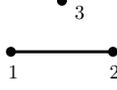}
   \begin{picture}(0,50)(0,50)
   \put(-9.5,65){$2$}
   \put(-27,95){$3$}
   \put(-60.5,65){$1$}
   \end{picture}}
     \caption{The canonical labeled join graph of three non-isomorphic congruence uniform lattices.}
        \label{join graph}
\end{figure}

We conclude this subsection with some examples of labeled and unlabeled graphs that can realized as the canonical join graph for some congruence uniform lattice.


\begin{example}[Complete graphs]\label{complete graphs}
\normalfont
In our first example we consider the complete graph $K_n$ on $n$ vertices, which can be realized as the canonical join graph for the boolean lattice $B_n$.
In fact, the boolean lattice is the only lattice whose canonical join graph is $K_n$.
\begin{proposition}\label{boolean}
Supposed that $L$ is a finite semidistributive lattice with canonical join graph equal to the complete graph $K_n$.
Then, $L$ is isomorphic to $B_n$.
\end{proposition}
\begin{proof}
Write $x_S$ for the element with canonical join representation $\Join (\{j_i: i\in S\}$, where $S\subseteq [n]=\{1,2\ldots,n\}$.
Suppose that $x_S \le x_{S'}$ for some $S'\subseteq [n]$, and there exists $k\in S$ that is not in $S'$.
Since $j_k \join \Join (\{j_i: i\in S'\}$ is a canonical join representation, in particular this join is irredundant.
So, $j_k \not \le \Join (\{j_i: i\in S'\} = x_{S'}$, and that is a contradiction.
Therefore, the map $x_S\mapsto S$ is order preserving.
It is immediate that the inverse map is order preserving.
\end{proof}
\end{example}

\begin{example}[Chordal graphs]
\normalfont
Similar to the construction of the complete graph (as a labeled canonical join graph), one can construct certain chordal graphs as the canonical join graph for a congruence uniform lattice.
In the construction, each doubling with respect to some interval $I$ has $I\times \2$ isomorphic to a boolean lattice.

Suppose that $G$ is a graph. 
The \newword{closed neighborhood} $N[v]$ is the subgraph of $G$ induced by the set of vertices $v'$ adjacent to $v$, together with $v$.
The \newword{open neighborhood} $N(v)$ is the subgraph induced by the set of vertices $v'$ adjacent to $v$.
A \newword{perfect elimination ordering} for $G$ is linear ordering $v_1<v_2<\cdots<v_n$ of the vertices of $G$ such that for each $i=1,2,\ldots, n$, the intersection of $N[v_i]$ with the set $\{v_i, v_{i+1}, \ldots, v_n\}$ is a clique in $G$.
Recall that a graph $G$ is \newword{chordal} if and only if it has a perfect elimination ordering.

\begin{proposition}
Suppose that $G$ is a labeled graph such that $\mathcal{L}=v_n<v_{n-1}<\ldots < v_1$ is a perfect elimination ordering.
If $N(v_{i+1}) \subseteq N(v_{i})$ for each $i\in [n-1]$, then there exists a congruence uniform lattice $L$ such that $G$ is its labeled canonical join graph.
\end{proposition}

\begin{proof}
We prove the statement by induction on $n$.
There exists a congruence uniform lattice $L'$ whose labeled canonical join graph is the subgraph induced by the first $n-1$ vertices.
In particular, $L'$ is isomorphic to $L''[I]$ where $L''$ is congruence uniform, $I= [a,b]$ is an interval in $L''$, and the vertex $v_{n-1}$ corresponds to the join-irreducible element $(a,1)$ in $L'$.

We give the argument for the case when that $v_n$ and $v_{n-1}$ are neighbors.
The proof is similar when $v_{n}\not\in N(v_{n-1})$.
We write $\{v_{i_1},\ldots, v_{i_k}\}$ for the set of vertices $N(v_n)\setminus \{v_{n-1}\}$, and $j_{i_1},\ldots, j_{i_k}$ for the corresponding join-irreducible elements of $L'$.
Since $\mathcal{L}$ is a perfect elimination order, the vertices $\{v_{i_1},\ldots, v_{i_k}, v_{n-1}\}$ form a clique in the subgraph induced by $V\setminus\{v_n\}$.
By Theorem~\ref{flag}, the join $(a,1)\join \Join\{j_{i_1},\ldots, j_{i_k}\}$  is a canonical join representation for some element $(y,1)$ in $L'$.

Consider the interval $I'=[(a,0),(y,1)]$.
It is straightforward (with Proposition~\ref{doubling}) to verify that the new join-irreducible element in $L'[I']$ joins canonically with each element in $\{j_{i_1},\ldots, j_{i_k}, (a,1)\}$.
Suppose that $(w,\epsilon')\in I'$ and $j\in L''$ such that $(j,\epsilon')$ is a canonical joinand of $(w, \epsilon)$ with $(j,\epsilon')\not\le (a,0)$.
We claim that $(j,\epsilon')$ corresponds to a vertex in the set $\{v_{i_1}, \ldots, v_{i_k}, v_{n-1}\}$.
The claim is obvious if $(j,\epsilon') = (a,1)$, so we assume that $j\ne a$.

First we show that $(j, \epsilon')$ is adjacent to $(a,1)$ in the canonical join graph for $L'$.
The last item of Proposition~\ref{consequences} implies that $y\in [a,b]$.
Therefore, $w$ is also in $[a,b]$.
The first item of Proposition~\ref{consequences} says that $j$ is a canonical joinand of $w$ (in $L''$), and the second item says that $\epsilon'=0$.
Therefore, $j\not \le a$.
Proposition~\ref{doubling} implies that $(j,0)$ is adjacent to $(a,1)$ in the canonical join graph of $L'$, as desired.

Finally, we show that $(j,\epsilon')=(j,0)$ belongs to the subset $\{j_{i_1},\ldots, j_{i_k}, (a,1)\}$ of neighbors of $(a,1)$.
The third and fourth items of Proposition~\ref{consequences} say that $(w,1)$ has $(j,0)$ and $(a,1)$ as canonical joinands.
Since $(w,\epsilon)$ is in $I'$, so is $(w,1)$.
In particular, $(j,0)\join (a,1)\in I'$.
Since $I'$ is a sublattice of $L'$, we have that the join $[(a,1)\join (j,0)] \join [(a,1)\join \Join\{j_{i_1},\ldots, j_{i_k}\}]$ also belongs to $I'$.
Therefore, $\Join (\{j_{i_1},\ldots, j_{i_k}, (a,1),(j,0) \})=(y,1).$
Because $\mathcal{L}$ is a perfect elimination ordering (and $(j,0)$ corresponds to a vertex $v_l$ with $l<n-1$), the set $\{j_{i_1},\ldots, j_{i_k}, (j,0), (a,1) \}$ is a face of the canonical join complex for $L'$. 
Therefore, $(j,0)$ is a canonical joinand of $(y,1)$, as desired.
The statement of the proposition now follows immediately from Proposition~\ref{doubling}.

The same argument, replacing the interval $[(a,0),(y,1)]$ with $[(a,1),(y,1)]$, proves the case in which $v_n$ and $v_{n-1}$ are not adjacent.
\end{proof}
\end{example}

\begin{example}[Cycle graphs]\label{cycle graphs}
\normalfont
For each positive integer $n$, we claim that there is a finite congruence uniform lattice whose canonical join graph is isomorphic to the unlabeled cycle graph $C_n$ on $n$ vertices.
We provide an illustration with examples for $n=5,6,7$.
Leftmost in Figure~\ref{C5} is the Hasse diagram for a distributive lattice $L$, and rightmost is the Hasse diagram obtained by doubling the interval $[a,e]$ in $L$.
(The middle Hasse diagram, which is isomorphic to the leftmost Hasse diagram, serves only to make the  doubling as clear as possible.)
Each distributive lattice is in particular congruence uniform, so the rightmost lattice is congruence uniform, as desired.
It is an easy exercise to verify that the canonical join graph for this right-most lattice is isomorphic to $C_5$.

The analogous construction is given in Figure~\ref{C6 and C7} for $n= 6$ and 7.
In these cases, the lattice $L$ being doubled is not distributive.
Because it is easy to check that $L$ is congruence uniform, we leave the details to the reader.
(Note that $C_n$, for $n\ge 5$ is among the minimal graphs excluded by Theorem~\ref{comparability_graph}, and so does not appear as the canonical join graph for a distributive lattice.)

\begin{figure}[h]
  \centering
\scalebox{.85}{\includegraphics{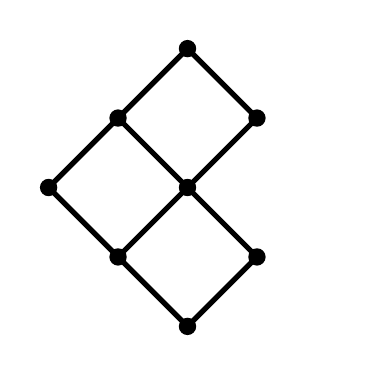}  \includegraphics{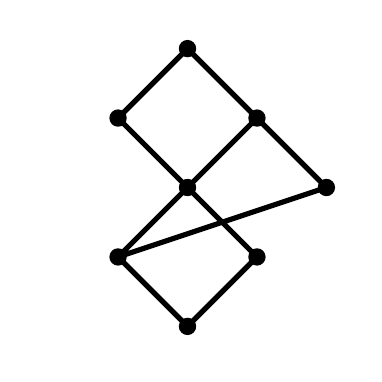} \includegraphics{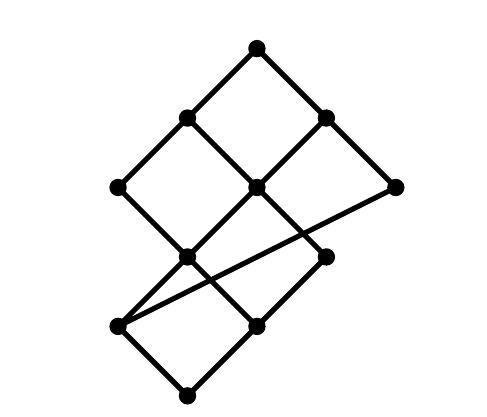}
    \begin{picture}(0,50)(0,50)
   \put(-350,82){$b$}
   \put(-370,106){$c$}
   \put(-290.5,82){$a$}
   \put(-348, 130){$d$}
   \put(-290.5, 130){$e$}
    \put(-233,87){$b$}
   \put(-155,106){$c$}
   \put(-178.5,87){$a$}
   \put(-233, 130){$e$}
   \put(-178.5, 130){$d$}
   \end{picture}}
   \caption{The two leftmost graphs are isomorphic Hasse diagrams for the distributive lattice $L$. Rightmost is the lattice obtained by doubling the interval $[a,e]$ in $L$.}
   \label{C5}
\end{figure}
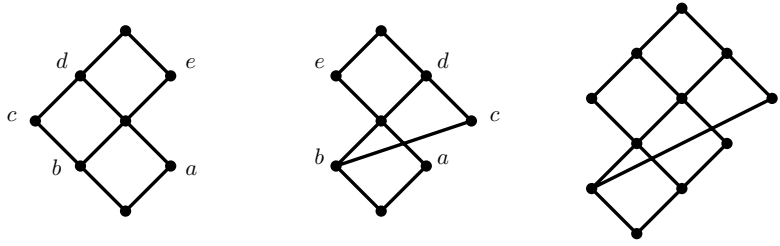 

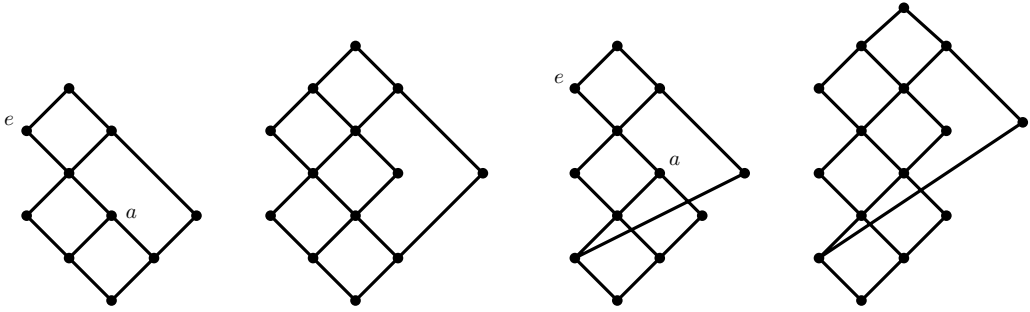
\begin{figure}[h]
  \centering
   \scalebox{.8}{\includegraphics{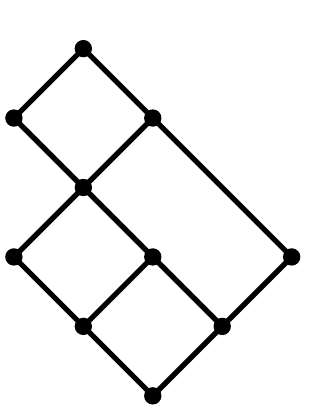}\qquad\includegraphics{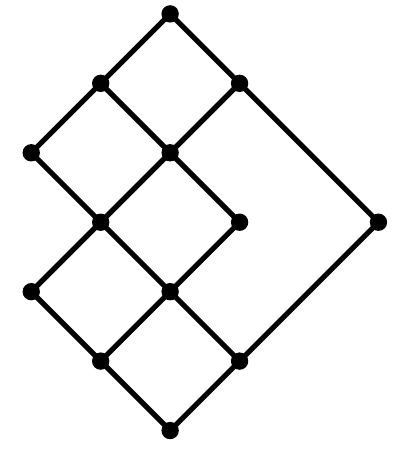}
   \includegraphics{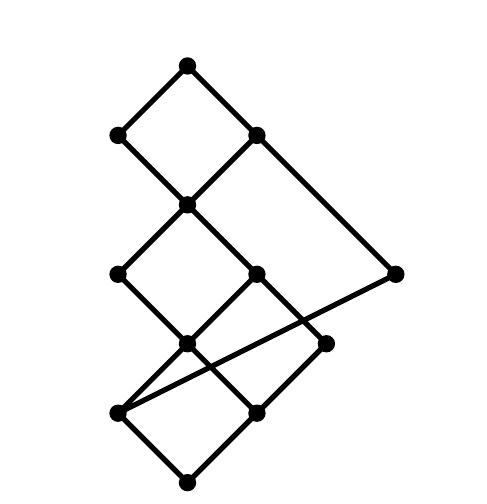}\includegraphics{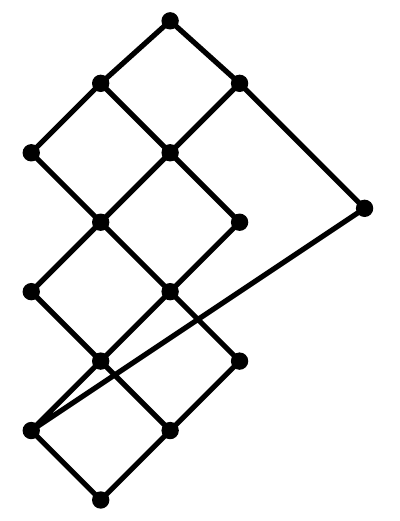}
     \begin{picture}(0,50)(0,50)
   \put(-437.5,95){$a$}
   \put(-495, 139){$e$}
    \put(-180.5,120){$a$}
   \put(-235, 159){$e$}
   \end{picture}}
   \caption{Doubling the interval $[a,e]$ in the leftmost congruence uniform lattice yields the left-middle lattice, whose canonical join graph is isomorphic to $C_6$.
   Doubling the interval $[a,e]$ in the right-middle lattice yields the rightmost lattice, whose canonical join graph is isomorphic to~$C_7$.}
   \label{C6 and C7}
\end{figure}

\end{example}

\section{Discussion and open problems}\label{discussion}
The discussion in Section~\ref{operations} does not constitute a complete list of lattice theoretic operations which preserve (join)-semidistributivity.
For example, the derived lattice $\mathbb{C}(L)$ discussed in \cite{derived}, the box product defined in \cite{box product} (see also, \cite[Corollary 8.2]{box product and jsd}), and the lattice of multichains from \cite{multichains} all preserve (join)-semidistributivity.

Because it is relatively easy, we will discuss this last operation in a small example.
Recall that an \newword{$m$-multichain} in a poset $\poset$ is a collection of $m$ elements satisfying $x_1\le x_2\le \ldots\le x_m$.
We write an $m$-multichain as a tuple $(x_1,\ldots, x_m)$ or more compactly as a vector $\vec{x}$.
We write the set of all $m$-multichains, partially ordered component-wise, as $\poset^{[m]}$.
When $\poset$ is a lattice, then $\poset^{[m]}$ is a sublattice of the $m$-fold direct product of $\poset^m$ (see \cite[Theorem~2.4]{multichains}).
It follows immediately that if $L$ satisfies~\ref{jsd} or ~\ref{msd} then $L^{[m]}$ also does, for each $m\in \mathbb{N}$ (see also, \cite[Proposition~2.10]{multichains}).
In the proposition below, $(j)_k$ is the element $(\0,\ldots,\0,j,\ldots,j)$, where $k$ is the left-most coordinate that is equal to $j$.
\begin{proposition}\label{join-irreducibles}
Suppose that $L$ is a finite lattice.
Then, $\Irr(L^{[m]})$ is equal to the set $\{(j)_k: j\in \Irr(L)\}$, where $k\in[m]$.
\end{proposition}
\begin{proof}
We first show that $(j)_k$ is join-irreducible when $j\in \Irr(L)$.
Suppose that $\vec{w}\join \vec{v} = (j)_k$.
We have $w_i \join v_i = j$, for each $i\ge k$.
Since $j$ is join-irreducible, we may assume that $w_k=j$.
Since $\vec{w}$ is a multichain, we have that $j\le w_i$ for each $i\ge k$.
Thus, $\vec{w}= (j)_k$, as desired.

Next, suppose that $\vec{w}\in \Irr(L^{[m]})$.
Let $w_k$ be the first nonzero entry in $\vec{w}$, and assume that $w_k\not \in \Irr(L)$ so that there exist $a$ and $b$ in $L\setminus \{w_k\}$ with $w_k = a\join b$.  
Then $\vec{w} = (\0,\ldots, \0, a, w_{k+1}, \ldots, w_m) \join (\0,\ldots, \0, b, w_{k+1},\ldots, w_m)$ in $L^{[m]}$.
By this contradiction, we conclude that $w_k\in \Irr(L)$.
Next, suppose that $w_i\ne w_k$, for some $i>k$.
Since $w_{k}< w_i$, there is an element $y\in \covdown(w_i)$ such that $w_{k}\le y$.
We have the following nontrivial join-representation of $\vec{w}$:
\[\vec{w} = (\0,\ldots, \0, w_k,\ldots, y, w_{i+1},\ldots, w_m) \join  (\0,\ldots,\0, y'',w_{k+1},\ldots, w_i,\ldots, w_m),\]
where $y''\in \covdown(w_{k})$.
Therefore $w_{i}=w_k$, and the proposition follows.
\end{proof}

\begin{example}
\normalfont
Let $L$ be the weak order on the symmetric group $S_3$, and consider $L^{[2]}$.
The lattice $L$ and $L^{[2]}$ are shown in Figure~\ref{multichain example}, and the corresponding canonical join complexes are shown in Figure~\ref{multichain complex}.
Observe that if $j\join j'$ is a canonical join representation in $L$ then both $(\0, j)\join (\0, j')$ and $(j, j) \join (j',j')$ are canonical join representations in $L^{[2]}$.
This accounts for the edges $\{(\0,a), (\0,b)\}$ and $\{(a,a), (b,b)\}$ in the complex for $L^{[2]}$.

To see how we obtain the remaining edges in Figure~\ref{multichain complex}, consider the canonical join representation of $(d,\1)$.
Observe that $\covdown((d,\1)) = \{(d,d), (b,\1)\}$.
It is easily checked that $(d, d)$ is the smallest element in $L^{[2]}$ whose join with $(b,\1)$ is equal to $(d, \1)$.
Similarly, $(\0,a)$ is the smallest element whose join with $(d, d)$ is equal to $(d, \1)$.
Therefore, the canonical join representation for $(d,\1)= (d,d)\join(\0,a)$.
The canonical join representations of the remaining elements in $L^{[2]}$ are computed similarly.

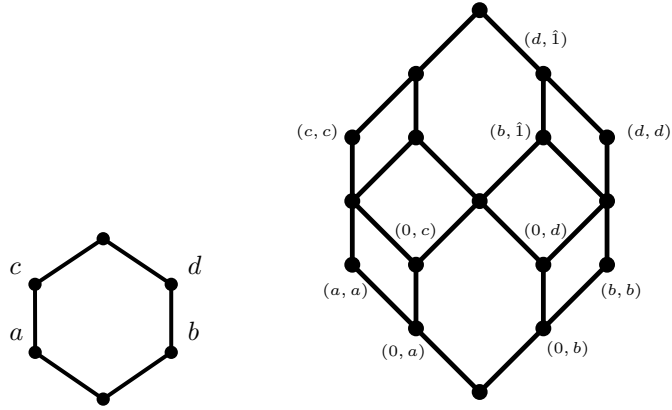
\begin{figure}[h]
  \centering 
  \includegraphics{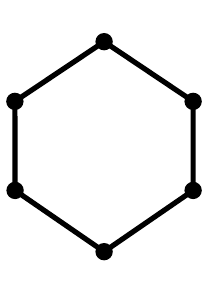}\qquad \qquad
   \scalebox{1.2}{\includegraphics{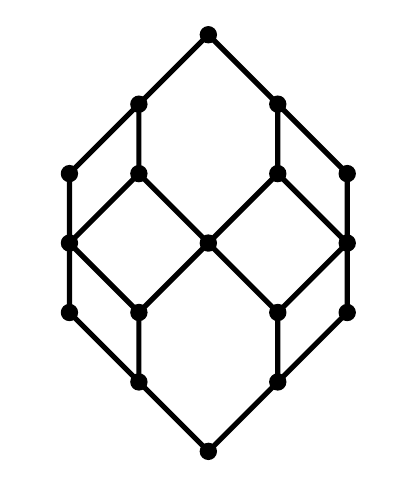}}
  \begin{picture}(0,50)(0,50)
    \put(-253.5,83){$a$}
    \put(-253.5,108){$c$}
   \put(-186,83){$b$}
   \put(-186,108){$d$}

   \put(-113,77){\tiny{$(0,a)$}}
    \put(-50.5,78){\tiny{$(0,b)$}}
   \put(-108,122.5){\tiny{$(0,c)$}}
   \put(-135,100){\tiny{$(a,a)$}}
   \put(-59,122.5){\tiny{$(0,d)$}}
   \put(-30,100){\tiny{$(b,b)$}}
   
    \put(-145,160){\tiny{$(c,c)$}}
     \put(-20,160){\tiny{$(d,d)$}}
   \put(-59,195){\tiny{$(d,\1)$}}
   \put(-72,160){\tiny{$(b,\1)$}}
   \end{picture}
   \vspace{5pt}
     \caption{Left: The weak order for the symmetric group $S_3$.
     Right: The lattice of 2-multichains.}
        \label{multichain example}
\end{figure}

\begin{figure}[h]
  \centering 
\begin{tikzpicture}
\draw [black, ultra thick] (0,0) -- (1,0);
\draw[fill] (0,0) circle [radius=0.07];
\draw[fill] (1,0) circle [radius=0.07];
\draw[fill] (0,1.1) circle [radius=0.07];
\draw[fill] (1,1.1) circle [radius=0.07];

\node [below] at (-0.23,-.05) {$a$};
\node [below] at (1.2,-.01) {$b$};
\node [below] at (-0.23,1.05) {$c$};
\node [below] at (1.2,1.07) {$d$};
\end{tikzpicture}
\qquad \qquad
\scalebox{1.1}{
\begin{tikzpicture}
\draw [black, ultra thick] (0,0) -- (1,0);
\draw [black, ultra thick] (0,0) -- (0,1);
\draw [black, ultra thick] (1,0) -- (1,1);
\draw [black, ultra thick] (0,1) -- (1,1);

\draw [black, ultra thick] (1,0) -- (1.5,-.5);
\draw [black, ultra thick] (0,0) -- (-.5,-.5);

\draw [black, ultra thick] (1,1) -- (1.5,1.55);
\draw [black, ultra thick] (0,1) -- (-.5,1.55);

\draw[fill] (0,0) circle [radius=0.07];
\draw[fill] (1,0) circle [radius=0.07];
\draw[fill] (0,1) circle [radius=0.07];
\draw[fill] (1,1) circle [radius=0.07];

\draw[fill] (1.5,1.55) circle [radius=0.07];
\draw[fill] (-.5,1.55) circle [radius=0.07];
\draw[fill] (1.5,-.5) circle [radius=0.07];
\draw[fill] (-.5,-.5) circle [radius=0.07];

\node [left] at (-.1,.05) {\tiny{$(\0, a)$}};
\node [right] at (1.1,.05) {\tiny{$(\0, b)$}};
\node [left] at (-0.05,1.05) {\tiny{$(b, b)$}};
\node [right] at (1.05,1.05) {\tiny{$(a,a)$}};

\node [right] at (1.6,1.6) {\tiny{$(\0,c)$}};
\node [left] at (-.5,-.5) {\tiny{$(d,d)$}};
\node [left] at (-.6,1.6) {\tiny{$(\0,d)$}};
\node [right] at (1.5,-.5) {\tiny{$(c,c)$}};
\end{tikzpicture}}

   \vspace{5pt}
     \caption{Left: The canonical join complex for weak order for the symmetric group $S_3$.
     Right: The canonical join complex for the lattice of 2-multichains.}
        \label{multichain complex}
\end{figure}
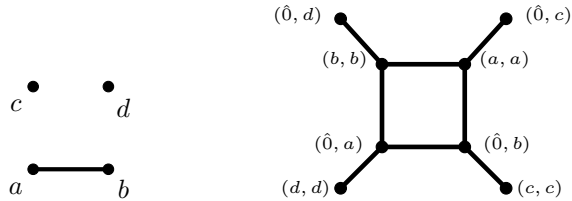
\end{example}

This example is emblematic of the general construction, as can be seen in the next proposition which describes the canonical join graph for $L^{[m]}$.
We leave the details of proof to the reader.
\begin{proposition}\label{cjr in multichains}
Suppose that $L$ is a finite semidistributive lattice with join-irreducible elements $j$ and $j'$.
\begin{enumerate}
\item If $i<k$ then $\{(j)_i, (j')_k\}$ is a face in the canonical join complex for $L^{[m]}$ if and only if $j'$ is a canonical joinand of $j\join j'$ in $L$.
\item If $i=k$, then $\{(j)_i, (j')_k\}$ is a face in the canonical join complex for $L^{[m]}$ if and only if $\{j, j'\}$ is a face in the canonical join complex for $L$.
\end{enumerate} 
\end{proposition}
Note that the operation on the canonical join complex corresponding to $L~\mapsto~L^{[m]}$ depends on the lattice $L$ (not just the canonical join complex for $L$).

 \begin{question}\label{question: operations}
What lattice theoretic operations (preserving join-semidistributivity) correspond to geometric operations on the canonical join complex that are independent of $L$?
\end{question}

Alternatively, it would be interesting to know which geometric operations (on the class of finite simplicial complexes) have a corresponding lattice theoretic analogue.
We point out that conspicuously absent from the discussion in Section~\ref{operations} is closure under taking induced subcomplexes (see Remark~\ref{forcing}). 

\begin{question}
Let $\mathcal{C}$ be the class of simplicial complexes that can be realized as the canonical join complex for some finite semidistributive lattice.
Is $\mathcal{C}$ closed under taking induced subcomplexes?
\end{question}

Say that $\G_n$ is the set of labeled graphs that can be realized the (labeled) canonical join graph for a congruence uniform lattice with $n$ join-irreducible elements, and $\G$ is the union $\bigcup_{n\in \mathbb{N}} \G_n$.
Using Stembridge's poset Maple package (\cite{StembridgePackages}) and Proposition~\ref{doubling}, we have counted the number of elements of $\G_n$ for $n\le 6$. 
While our computations indicate that not every labeled graph appears, they also suggest that $\G$ is closed under subgraphs (so that the corresponding class of simplicial complexes is closed under taking subcomplexes).
We close the paper by asking two related questions:
\begin{question}\label{question: graphs}
Which labeled graphs can be realized as the labeled canonical join graph for some congruence uniform lattice?
\end{question}
\begin{question}\label{question: iso}
Suppose that $G$ is the canonical join graph for a congruence uniform lattice $L$.
What data, in addition to $G$, is necessary in order to determine $L$ up to isomorphism?
\end{question}

\section{Acknowledgements}
The author thanks Patricia Hersh (who suggested the connection to the crosscut complex), Thomas McConville, Nathan Reading and Victor Reiner for many helpful suggestions.


\end{document}